\documentclass[11pt,a4paper]{amsart}

\usepackage{amsmath}
\usepackage{amstext}
\usepackage{amssymb}
\usepackage{amsthm}
\usepackage{enumerate}
\usepackage[shortlabels]{enumitem}

\usepackage[utf8]{inputenc}
\usepackage[T1]{fontenc}

\usepackage{todonotes}

\usepackage[pdfauthor={Fischer, Friedman, Schrittesser, Tornquist},
            pdftitle={Definable maximal cofinitary groups of intermediate size},
            pdfproducer={Latex with hyperref},
            pdfpagemode=UseOutlines]{hyperref}

\makeatletter
\def\imod#1{\allowbreak\mkern10mu({\operator@font mod}\,\,#1)}
\makeatother

{\end{minipage}\end{equation*}}

\newtheorem{theorem}{Theorem}[section]

\newtheorem{prop}[theorem]{Proposition}
\newtheorem{lemma}[theorem]{Lemma}
\newtheorem{claim}[theorem]{Claim}

\newtheorem{assumption}[theorem]{Assumption}
\newtheorem{induction}[theorem]{Properties of the Induction}

\theoremstyle{definition}
\newtheorem{definition}[theorem]{Definition}
\newtheorem{notation}[theorem]{Notation}
\newtheorem{example}[theorem]{Example}

\theoremstyle{remark}
\newtheorem{remark}[theorem]{Remark}

\theoremstyle{remark}

\numberwithin{equation}{section}

\providecommand{\res}{\mathbin{\upharpoonright} }

\DeclareMathOperator{\set}{use}

    \DeclareMathOperator{\limord}{lim}

    \DeclareMathOperator{\lh}{lh}

    \DeclareMathOperator{\dom}{dom}
    
    \DeclareMathOperator{\ran}{ran}

    \DeclareMathOperator{\im}{im}

    \DeclareMathOperator{\fix}{fix}

    \DeclareMathOperator{\powerset}{\mathcal{P}}

\newcommand{\ON}{\mathrm{On}}

\newcommand{\NS}{\mathsf{NS}}

\newcommand{\rest}{{\upharpoonright}}

    \newcommand{\forces}{\Vdash}

    \newcommand{\la}{\langle}
    \newcommand{\ra}{\rangle}

\def\M{{\mathcal M}}
\def\R{{\mathbb R}}

\def\N{{\mathbb N}}
\def\Z{{\mathbb Z}}

\def\Q{{\mathbb Q}}

\def\PP{{\mathbb P}}
\def\QQ{{\mathbb Q}}

\def\DD{{\mathbb D}}

\def\NN{{\mathbb N}}

\newcommand{\cW}{\mathcal{W}}
\newcommand{\Wa}{\cW^{\textup{a}}}
\newcommand{\Wncs}{\cW^{\textup{ncs}}}
\newcommand{\Wnr}{\cW^{\textup{nr}}}

\newcommand{\ZFC}{\textup{\textsf{ZFC}}}
\newcommand{\ZF}{\textup{\textsf{ZF}}}

\date{\today}
\begin{document}

\title[Good Projective Witnesses]{Good Projective Witnesses}

\author[Fischer]{Vera Fischer}

\address{Vera Fischer, University of Vienna, Kolingasse 14--16, 1090 Vienna, Austria}
\email{vera.fischer@univie.ac.at}

\author[Friedman]{Sy David Friedman}
\address{Sy David Friedman, University of Vienna, Kolingasse 14--16, 1090 Vienna, Austria}
\email{sdf@logic.univie.ac.at}

\author[Schrittesser]{David Schrittesser}
\address{David Schrittesser, Institute for Advanced Study in Mathematics,Harbin Institute of Technology, 150001, Heilongjian, China
\emph{and}} 
\address{Suzhou Research Institute, Harbin Institute of Technology, 215104, Jiangsu, China}
\email{david@logic.univie.ac.at}

\author[T\"ornquist]{Asger T\"ornquist}

\address{Asger T\"ornquist, Department of Mathematical Sciences, University of Copenhagen, Universitetspark 5, 2100 Copenhagen, Denmark}
\email{asgert@math.ku.dk}



\begin{abstract} 
We develop a new forcing notion for adjoining self-coding cofinitary permutations and use it to show that consistently, the minimal cardinality $\mathfrak a_{\text{g}}$ of a maximal cofinitary group (MCG) is strictly between $\aleph_1$ and $\mathfrak{c}$, and there is a $\Pi^1_2$-definable MCG of this cardinality. Here $\Pi^1_2$ is optimal, making this result a natural counterpart to the Borel MCG of Horowitz and Shelah.
Our theorem has its analogue in the realm of maximal almost disjoint (MAD) families, extending a line of results regarding the definability properties of MAD families in models with large continuum.
\end{abstract}

\maketitle

\section{Introduction}

We will be interested in subgroups of $S_\infty$, the group of all permutations of the natural numbers, with the additional property that all of their non-identity elements have only finitely many fixed points. Such groups are referred to as \emph{cofinitary groups}, while permutations which have only finitely many fixed points are referred to as \emph{cofinitary permutations}. A cofinitary group which is not properly contained in another cofinitary group, is called a \emph{maximal} cofinitary group, abbreviated \emph{MCG}. 
One way to see the existence of MCGs is to use the Axiom of Choice (short, \textsf{AC}), which leaves many questions open regarding their possible cardinalities and their descriptive set-theoretic definability.
That a (Borel) MCG can be constructed without appealing to \textsf{AC} is a relatively recent result, 
obtained by Horowitz and Shelah \cite{HHSS}  in 2016.

\medskip

The study of the {\emph{the spectrum}} of maximal cofinitary groups, i.e. 
of the set of different sizes of MCGs, 
 \[
 \operatorname{spec}(\mathrm{MCG}):=\{|\mathcal{G}|: \mathcal{G}\text{ is a maximal cofinitary group}\}
 \]
has been of interest since the early development of the subject. 
It was shown by Adeleke~\cite{Adeleke} that every maximal cofinitary groups is uncountable,
Neumann showed that there is always a maximal cofinitary group of size $\mathfrak{c}$, 
and Zhang~\cite{Zhang} showed whenever $\omega<\kappa\leq\mathfrak{c}$, consistently there is a maximal cofinitary group of size $\kappa$. 
A systematic study of $\operatorname{spec}(\mathrm{MCG})$ is found in~\cite{JBOSYZ}, a study which was later generalized to analyze the higher analogue of maximal cofinitary groups in $S_\kappa$ (see~\cite{VF}; here $\kappa$ is an arbitrary regular uncountable cardinal and $S_\kappa$ denotes the group of permutations of $\kappa$).  
In~\cite{VFAT} it was shown to be consistent that the minimum of $\operatorname{spec}(\mathrm{MCG})$, denoted $\mathfrak{a}_g$, be of countable cofinality.

\medskip

For the following discussion, let us make two definitions:

\begin{definition}\label{def.witness} We refer to maximal cofinitary groups of cardinality $\mu$, as {\emph{witnesses}} to $\mu\in\operatorname{spec}(\mathrm{MCG})$ and to values $\mu\in\operatorname{spec}(\mathrm{MCG})$ such that $\aleph_1<\mu<\mathfrak{c}$ as 
{\emph{intermediate cardinalities (or values)}}.
\end{definition}

The only known ways of constructing MCGs are \textsf{AC}, forcing, and the Horowitz-Shelah construction (\cite{HHSS}). 
It is therefore interesting to ask: 
Given a cardinal $\mu$, what are the possible \emph{definability properties} of a witness to $\mu\in\operatorname{spec}(\mathrm{MCG})$?

\begin{definition}\label{def.good.projective}  {\emph{A good projective witness}} to $\mu\in\operatorname{spec}(\mathrm{MCG})$  %
is a MCG $\mathcal G$ %
of cardinality $\mu$ which is also of lowest projective complexity (i.e., there is no witness to $\mu\in\operatorname{spec}(\mathrm{MCG})$ whose definitional complexity lies strictly below that of $\mathcal G$ in terms of the projective hierarchy).
\end{definition}

It is easy to see that if a MCG is $\Sigma^1_n(r)$, it is $\Delta^1_n(r)$. 
Gao and Zhang (see~\cite{SGYZ}) showed that in $L$, there exists a MCG of size $\omega_1$ with a co-analytic set of generators, a result which was later improved by Kastermans~\cite{BK}, who showed that in $L$ there is a co-analytic MCG.  
The first, third, and fourth authors found a co-analytic MCG in $L$ which remains maximal after adding Cohen reals \cite{VFDSAT}. 
This showed that in a generic extension of $L$, there is a $\Pi^1_1$ good projective witness to $\omega_1\in\operatorname{spec}(\mathrm{MCG})$.\footnote{In this model $2^{\aleph_0}>\aleph_1$, ruling out a Borel witness to $\aleph_1\in\operatorname{spec}(\mathrm{MCG})$.}

For a long time existence of analytic (equivalently, Borel) MCGs was one of the most interesting open questions in the area, a question which was answered affirmatively by the beforementioned construction due to Horowitz and Shelah in~\cite{HHSS}.\footnote{The second author of the present paper later improved their result.} 
By their result, there is a Borel witness to $\mathfrak c \in \operatorname{spec}(\mathrm{MCG})$.

\medskip

So far the study of definable witnesses to $\mu\in\operatorname{spec}(\mathrm{MCG})$ has concentrated on either $\mu=\aleph_1$ or $\mu=\mathfrak{c}$; 
nothing is known about the definability properties of witnesses to intermediate cardinalities. 
The present paper is motivated by the question: {\emph{What can we say about the definability properties of maximal cofinitary groups $\mathcal{G}$ such that $\aleph_1<|\mathcal{G}|<\mathfrak{c}$?}}  

Here is a first answer:
\begin{theorem}\label{t.L(R)}
It is relatively consistent with $\ZFC$ that $\mathfrak c\geq\aleph_3$ and there is an MCG $\mathcal G$ of size $\aleph_2$, $\mathcal G\in L(\R)$.
\end{theorem}
In other words, consistently there exists a witness to an intermediate value  in $L(\R)$. 
While we sketch a proof of a slight strengthening of Theorem~\ref{t.L(R)} in Section~\ref{section.group.forcing} for expository purposes, for an optimal answer we must find a model with a good projective witness.
For this, first note that a $\mathbf{\Sigma}^1_2$ maximal cofinitary group must be either of size $\aleph_1$ or continuum (being the union of $\aleph_1$ many Borel sets). By this observation, the lowest possible projective complexity of a
witnesses to intermediate values in $\operatorname{spec}(\mathrm{MCG})$ is $\Pi^1_2$. 

\medskip

Our main theorem is the following:
\begin{theorem}\label{t.main} 
It is relatively consistent with $\ZFC$ that $\mathfrak c \geq\aleph_3$ and there exists a $\Pi^1_2$ MCG  of size $\aleph_2$.
Thus, it is consistent that there is a $\Pi^1_2$ good projective witness to an intermediate value in
$\operatorname{spec}(\mathrm{MCG})$.
\end{theorem}

In fact, we show:

\begin{theorem}\label{t.precise} Let $2\leq M < N < \aleph_0$ be given. There is a cardinal preserving generic extension of the constructible universe $L$ in which $$\mathfrak{a}_g=\mathfrak{b}=\mathfrak{d}=\aleph_M<\mathfrak{c}=\aleph_N$$
and there is a $\Pi^1_2$ definable maximal cofinitary group of size $\aleph_M$.
\end{theorem}

\noindent
(The cardinal characteristics $\mathfrak b$ and $\mathfrak d$ referred to in the above theorem are the \emph{bounding number} and the \emph{dominating number}; for an introduction to cardinal characteristics, see \cite{BLASS}).

\begin{remark} Providing a model in which there is a maximal cofinitary group of cardinality $\mu$ where $\aleph_1<\mu<\mathfrak{c}$ and 
$\aleph_\omega<\mathfrak{c}$, or even $\aleph_\omega\leq\mu$, is possible, but our proof would be considerably more technical since it uses Jensen coding. 
For the sake of clarity and brevity we have chosen to restrict our work to values of the continuum below $\aleph_\omega$.
\end{remark}

While much of the proof in this article also applies to the case $M=1$ in  Theorem~\ref{t.precise}, 
one shold not expect to produce a model with a good projective witness to $\aleph_1\in\operatorname{spec}(\mathrm{MCG})$ in this way. 
As has been mentioned, the consistency of  such a witness was already shown by the first, third, and fourth authors in~\cite{VFDSAT} by constructing a Cohen-indestructible $\Pi^1_1$ MCG in $L$; in this model $\mathfrak{a}_g=\mathfrak{d}=\aleph_1<\mathfrak{c}$. 
The consistency of $\mathfrak{a}_e=\mathfrak{d}=\aleph_1<\mathfrak{c}$ with a good projective witness to $\mathfrak{a}_e$ is proved in \cite{VFDS} ($\mathfrak a_e$ is the smallest size of a maximal eventually different family).

\medskip

Let us extend our terminology to MAD families by writing
 \[
 \operatorname{spec}(\mathrm{MAD}):=\{|\mathcal{A}|: \mathcal{A}\text{ is a MAD family}\}
 \]
and let us speak of  \emph{witness} and \emph{good projective witness to} $\mu\in\operatorname{spec}(\mathrm{MAD})$, and \emph{intermediate cardinalities}, with the obvious meaning analogous to Definitions~\ref{def.witness} and~\ref{def.good.projective}. 

Studies of the definability properties of maximal almost disjoint families can be found in~\cite{JBYK, VFSFYK, SFLZ, AT}. 
With the exception of~\cite{AT}, in all of these studies the maximal almost disjoint family of interest is of cardinality $\mathfrak{c}$.

Our techniques easily modify to the study of maximal almost disjoint families 
and provide the following result:

\begin{theorem} Let $2\leq M < N < \aleph_0$ be given. There is a cardinal preserving generic extension of the constructible universe $L$ in which $$\mathfrak{a}=\mathfrak{b}=\mathfrak{d}=\aleph_M<\mathfrak{c}=\aleph_N$$
and the MAD family witnessing $\aleph_M\in\operatorname{spec}(\mathrm{MAD})$ is $\Pi^1_2$, i.e., it is a good projective witness to the intermediate cardinality $\aleph_M\in\operatorname{spec}(\mathrm{MAD})$.
\end{theorem}

A good witness to   $\mathfrak{c}\in\operatorname{spec}(\mathrm{MAD})$ is constructed by Brendle and Khomskii  in~\cite{JBYK}, while a Cohen indestructible co-analytic maximal almost disjoint family in $L$ is a good witness to $\aleph_1\in\operatorname{spec}(\mathrm{MAD})$. The study of projective witnesses does not limit to MCGs and mad families. Let $\operatorname{spec}(\mathrm{IND})$ denote the set of possible cardinalities of maximal independent families. One of the main results of~\cite{JBVFYK} shows that $\aleph_1\in\operatorname{spec}(\mathrm{IND})$ has a good projective witness, while the existence of a good projective 
witness to $\mathfrak{c}\in\operatorname{spec}(\mathrm{IND})$ is still open.

\medskip

Let us say a word about the methods used in this paper.
Two techniques had to be devised: 
\begin{enumerate}
	\item  A mechanism which adjoins a MCG, or more generally, enlarges a MCG by forcing in such a way that each new group element ``codes'' a pre-ordained subset of $\omega$; this is the content of Section~\ref{section.group.forcing}.
	\item A way to ensure that the definition of the resulting maximal cofinitary group is of minimal complexity in the projective hierarchy; for this, both Section~\ref{section.group.forcing} as well as the entire remainder of the paper are relevant.
\end{enumerate}
A very special case of the first problem was solved in \cite{VFDSAT}, namely the case where the group to be enlarged is countable, or equivalently, the group to be adjoined has size at most $\aleph_1$ 
(this built on previous work in \cite{Zhang,VFAT}, which describe forcings to adjoin MCGs but without any coding requirement). 
In this special case, what we call ``coding paths'' can always be taken to be disjoint, which simplifies the forcing immensely. In this paper, we solve the problem without this cardinality restriction. 
This requires a restriction on the type of group which can be enlarged. 

\medskip

To solve the second problem, we use a technique, originally inspired by \cite{DAVID}, of adjoing reals to make a given set projective.
The groundwork for this approach was laid for \cite{memoirs} and it has previously been employed in different contexts, e.g., in \cite{VFSF,VFSFLZ11}.
Making use of this approach in the present context is not straightforward: 
We must 
simultaneously guarantee maximality of our group, and that the group's elements code reals which make its definition projective.
This difficulty was eventually resolved by a very careful arrangement of the entire forcing iteration, and by using generic eventually different families (see the discussion at the beginning of Section~\ref{section.group.forcing} and the road-map given at the beginning of Section~\ref{s.iteration}).

\medskip

These ideas lend a flexibility to our construction without which
our final goal could not be achieved. 
They also present promising and robust techniques to address existing open problems. Some of the many naturally occurring remaining open questions are discussed in our final section.

\subsection*{Structure of the paper}
Section \ref{section.group.forcing} presents a new forcing notion which can adjoin self-coding permutations to a given cofinitary group. Section \ref{s.iteration} presents the entire forcing construction leading to our main result. Our main result is established in Section \ref{s.definability.maximality}. Some remaining open problems are listed in Section \ref{s.questions}.

\subsection*{Acknowledgments}

The first author would like to thank the Austrian Science Fund (FWF) for the generous support through START Grant Y1012-N35. The second author would also like to thank the FWF for its generous support through projects P25748 and I1921. The third author thanks the FWF for its support through project P29999 as well as the Basic Research Program of Jiangsu for support through project BK20241788 (``Non-Standard Analysis and Its Applications in Statistics and Economics'').
The fourth author gratefully acknowledges support from DFF (Independent Research Fund Denmark) through the project ``Operator Algebras, Groups, and Quantum Spaces''. %

\section{Adding cofinitary groups of coding permutations}\label{section.group.forcing}

In this section, we introduce a forcing $\QQ$ which enlarges a (certain type of) cofinitary group from the ground model to a larger cofinitary group by adding a single generic permutation $\sigma^G$ as a new generator, while at the same time ensuring that particular sets (from the ground model) are constructible from \emph{each} new group element.

\medskip

First, why do we need this forcing? 
To simplify the discussion, assume we aim to find a model with a MCG $\mathcal G$ of size $\mu = \omega_2 < 2^\omega$ and such that $\mathcal G$ is \emph{definable in $L(\R)$ without parameters} (instead of projectively definable).

\begin{notation}
For the rest of this paper, let us fix a computable bijection 
\begin{equation}\label{e.comp.bij}
\psi:\omega\times\omega\to\omega.
\end{equation}
\end{notation}
Let us suppose until the end of this section that already $2^\omega > \mu$; and that we have at our disposal a definable collection 
\[
\la S_{\xi,m} \colon \xi< \mu,m<\omega \ra \in L
\] 
of subsets of $\omega_1$ which are stationary in $L(\R)$ but not in $V$.
More precisely, let us assume that for each $(\xi,m)\in\mu\times\omega$ there is $C_{\xi,m} \in V$ such that  
\[
L(\R)[C_{\xi,m}]\vDash S_{\xi',m'} \in \NS \iff (\xi',m')=(\xi,m) 
\]
(this is not hard to arrange; see Section~\ref{s.iteration}).

Our aim is that $\mathcal G$ be definable as follows: $g \in \mathcal G$ if and only if
\begin{equation}\label{e.def}
(\exists \xi < \mu)(\forall m\in\omega) \; \left(m\in\Psi[g] \iff L[g]\vDash S_{\xi,m}\text{ is \emph{not} stationary}\right).
\end{equation}
To create this model, we face the following problem:
Given a cofinitary group $\mathcal G_0$, find a cardinality preserving forcing $\QQ$ which adds a cofinitary $\sigma\in S_\infty$
such that
\begin{enumerate}[label=(\Alph*),ref=\Alph*]
\item\label{i.perm} the group $\mathcal G_1$ generated by $\mathcal G_0\cup\{\sigma\}$ is cofinitary and maximal with respect to the ground model, i.e., for no $f \in S_\infty \cap V$ is $\mathcal G_0\cup\{\sigma,f\}$ cofinitary;
\item\label{i.code} For all $g\in S_\infty$, it holds that $g \in \mathcal G_1\setminus \mathcal G_0\iff$ \eqref{e.def}.
\end{enumerate}
To simplify the combinatorial properties of $\QQ$ we choose to replace \eqref{i.code} by the following (see \ref{r.subwords} for details):
\begin{enumerate}[label=(\Alph*$'$),ref=\Alph*$'$,start=2]
\item\label{i.code'} For all $g\in S_\infty$, it holds that $g \in \mathcal G'_1\iff$ \eqref{e.def}, where $\mathcal G'_1\subsetneq \mathcal G_1\setminus \mathcal G_0$ is a specifically chosen, ``sufficiently large'' subset.
\end{enumerate}
We may then iterate $\QQ$ (taking for $\mathcal G_0$ the group generated by generic permutations added at previous steps of the iteration, starting with the trivial group) to length $\mu$ and obtain the desired model.

A forcing that will achieve Item \eqref{i.perm} was invented by Zhang \cite{Zhang}.
For Item \eqref{i.code'}, in particular for $\Rightarrow$ in \eqref{e.def},
we want that for each $g\in\mathcal G'_1$ there is $\xi=\xi(g) <\mu$ such that $C_{\xi,m}\in L[g]$, i.e., ``$C_{\xi,m}$ is coded by $g$'', for each $m \in\Psi[g]$.
\begin{remark}\label{r.conditional}
Our construction will ensure that $S_{\xi,m}$ remains stationary in $L(\R)$ when $m \notin \Psi[g]$ (see Lemma~\ref{no_accidental_real}). This is how we will show  $\Leftarrow$ in \eqref{e.def}, and thus, Item~\eqref{i.code'} (in Lemma~\ref{lemma1}). 
For this it is essential that coding $C_{\xi,m}$ by a real (in fact, by $g$) is conditional on $\Psi^{-1}(m) \in g$ and so it must be done simultaneously with adding $g$.
\end{remark}
As part of our solution, we define $\QQ$ as a hybrid between Zhang's forcing and Solovay's almost disjoint coding: 
Each
$g\in \mathcal G'_1$ will code a sequence $\la Y_{\xi(g),m}: m \in \Psi[g]\ra$ using almost-disjoint coding with respect to an eventually different family $\mathcal F$ of permutations.
For the present discussion, the reader may assume $Y_{\xi,m}=C_{\xi,m}$ 
(in the next section we discuss how to build $Y_{\xi,m}$ to achieve $\mu>\aleph_2$ and a projective definition of $\mathcal G$).

For the above strategy to succeed, we need a family $\mathcal F$ with particular properties: Firstly, for the density argument below in Lemma~\ref{l.generic.hitting} we need that for any $f \in \mathcal F$, $\mathcal G_0\cup\{f\}$ is cofinitary.
Secondly, to obtain maximality relative to the ground model in Item \eqref{i.perm} we must choose $\mathcal F \cap V=\emptyset$. 
The second property means that $g$ also has to ``code'' $\mathcal F$ in some way, since we want that $Y_{\xi,m}$ is constructible relative to only $g$.
We can use Solovay's forcing to ensure $\mathcal F \in L[c^{\mathcal F}]$ for some $c^{\mathcal F}\in\powerset(\omega)$ and add the following to the list of tasks for our forcing, where $z=c^{\mathcal F}$:
\begin{enumerate}[label=(\Alph*),ref=\Alph*,start=3]
\item\label{i.computable.coding} For all $g\in \mathcal G'_1$, a pre-ordained real $z$ is computable from $g$.\footnote{In fact, the problems  which lead us to restrict to $\mathcal G'_1$ arise precisely from the coding demand in \eqref{i.computable.coding}; see \ref{r.subwords}.}
\end{enumerate}
For this, we must make some kind of assumption on $\mathcal G_0$ (see Assumption~\ref{a.sufficiently.generic}). 
In fact, as it turns out, it is possible to specify a \emph{different} real, $z=z^g$, to be computable in $g$, for each $g\in \mathcal G'_1$.
While we could make do by taking $z = c^{\mathcal F}$ for every $g \in \mathcal G'_1$, allowing $z^g$ to vary with $g$ comes at almost no additional effort and it will prove to be convenient (see \eqref{e.z^w}).

\subsection{Finite partial extensions of cofinitary groups}\label{s.finite.partial}

As indicated above, we wish to adjoin a new generator, $\sigma \in S_\infty$ to extend a cofinitary group $\mathcal G_0\leq S_\infty$. 
This will be done using a forcing $\QQ$ whose conditions contain,  among other information, finite partial injective functions $s \colon \omega \rightharpoonup \omega$ approximating $\sigma = \sigma^G$ ($G$ is the $\Q$-generic set).
An injective partial function $s:\NN\rightharpoonup\NN$ will be referred to as a \emph{partial permutation}. 

The group $\mathcal G_0$ will naturally come with its own set of generators.
Each partial permutation $s \colon \omega \rightharpoonup \omega$ then defines a monoid (a set with an associative binary operation and a two-sided identity) extending $\mathcal G_0$.
Some terminology will be useful.

\medskip

Given a set $A$ (the `index set') write $\mathbb{F}(A)$ for the free group with generating set $A$. 
A convenient presentation of $\mathbb{F}(A)$ is as the reduced words $W_A$ in the alphabet $A\cup A^{-1} := \{a^{i} : a \in A, i\in\{-1,1\}\}$ where the group operation is ``concatenate and reduce'' and the unit is the empty word $\emptyset$ (see, e.g., \cite[Normal Form Theorem]{lyndon-combinatorial}).

\medskip

For $w_0, w_1 \in W_{A}$ we say $w_1$ is a \emph{proper conjugate subword of} $w_0$ if $w_0 = w^{-1} w_1 w$ for some word $w \in W_{A}\setminus\{\emptyset\}$ and $w_1 \neq \emptyset$.
We say $w_0$ is a \emph{root} of $w_1$ if $w_0 \neq w_1$ and $w_1 = (w_0)^n$ for some $n\in\omega$  (so that $n>1$, one might add, without changing the definition).

Write $w \sqsupseteq w'$\label{d.sqsupseteq} to mean that $w'$ is a \emph{right-initial segment} (or initial segment from the right) of $w$, i.e., $w = a_n\hdots a_0$ and $w' = a_k\hdots a_0$ with $0\leq k \leq n$, or $w'=\emptyset$.

Similarly, we say that $w'$ is an \emph{left-initial segment} (or initial segment from the left) of $w$ to mean $w = a_n\hdots a_0$ and $w' = a_n\hdots a_k$ with $0\leq k \leq n$, or $w'=\emptyset$.

\medskip

Another naturally appearing notion is that of a \emph{circular shift (with offset $k$)} of a word (see~\cite{VFDSAT}) in $W_A$. More precisely, given such a word $w=w_n\cdots w_1$, where $w_i=a_i^{j_i}$, $j_i\in\{-1,1\}$ with $a_i \in A$ for each $i$, and a permutation $\sigma:\{1,\cdots, n\}\to\{1,\cdots,n\}$ such that
$\sigma(i)=i + k\hbox{ mod }n$ for some $k\in \NN$, we will refer to $w_{\sigma(n)}\cdots w_{\sigma(1)}$ as a circular shift (with offset $k$) of $w$. Thus, in particular, for each $n$ there are only finitely many circular shifts of a given word.

\medskip

For $X \subseteq S_\infty$, we write $\la X\ra$ for the subgroup of $S_\infty$ generated by $X$.
We call an injection $\rho: A\to S_\infty$ such that $\la\im(\rho)\ra$ is a cofinitary group, a \emph{cofinitary representation}.\label{d.cofinitary.representation} 
Such a map obviously gives rise to a group homomorphism $\mathbb{F}(A) \to S_\infty$, denoted also by $\rho$.

\medskip

Given some $a$ such that $\{a,a^{-1}\} \cap A = \emptyset$, let us write $W_{A,a}$ for $W_{A\cup\{a\}}$. 
Given a word $w\in W_{A,a}$ and a (possibly partial) permutation $s$ we denote by $w[s]$ the (possibly partial) permutation $w[s]\colon \omega \rightharpoonup\omega$ obtained by substituting each occurrence of $x^j$ where $x\in A$ and $j\in\{-1,1\}$ with $\rho(x)^j$ and $a^j$ where $j\in\{-1,1\}$ with $s^j$. 

\medskip

Let us make precise the notion of \emph{path}: Given a word $w\in W_{A,a}$ and writing $w=a_n^{j_n}\cdots a_1^{j_1}$, where $j_i\in\{-1,1\}$ and $a_i \in A\cup\{a\}$ for $1 \leq i\leq n$, and given a (possibly partial) permutation $s$, \emph{the path} of a given integer $m$ under $(w,s)$ is the sequence $\la m_k:k\in\alpha\ra$, where $m_0=m$; and for each $k$ such that $k=nl+i$ with $i<n$, 
$$m_k=(a_i^{j_i}\cdots a_1^{j_1}w^{nl})[s](m),$$
and where $\alpha$ is either $\omega$, or denotes $k+1$ where $k$ is maximal so that $m_{k}$ as above is defined.
In the latter case we say the path has \emph{terminating value} or \emph{terminating point} $m_{\alpha-1}$.
We shall also say that $a^{j_i}$ is the \emph{letter applied at step $k-1$}, when $k>0$.

Sometimes it suffices to think of the path merely as a set, rather than as a sequence; so let
\[
\set(w,s,m)=\{ m_i \colon i < \alpha \}.
\]

\medskip

Finally, we shall frequently refer to the following subsets of $W_{A,a}$:
\begin{notation}\label{n.words}
Let us write 
\begin{itemize}
\item  $\Wa$ for the set of words from $W_{A,a}$ in which $a$ or $a^{-1}$ occurs at least once,
\item $\Wncs$ for the set of words from $\Wa$ without any proper conjugate subwords, 
\item $\Wnr$ for the set of words $w\in\Wa$ without any \emph{roots}, i.e., so that there is no $w' \in \Wa\setminus\{w\}$ and $m\in \omega$ with $w = (w')^m$.
\end{itemize}
\end{notation}

\begin{remark}\label{r.subwords}
From now on, when speaking of the computable coding (as in \eqref{i.computable.coding}, p.~\pageref{i.computable.coding}) of a real $z^g$ by $g$ for each $g \in \mathcal G'_1$, let us write just $z^w$ instead of $z^g$ with $g=w[\sigma^G]$.
We are now also able to give a precise definition of $\mathcal G'_1$ (used in \eqref{i.code'} and \eqref{i.computable.coding} above): 
We let
$\mathcal G'_1 = \{w[\sigma^G] : w\in\Wncs\cap\Wnr\}$. 
While this simplifies the computable coding of $z^w$ by $w[\sigma^G]$,
we nevertheless obtain a $\Pi^1_2$ defininition of $\mathcal G_1 \setminus \mathcal G_0$, because each element of $\mathcal G_1 \setminus \mathcal G_0$ will be projectively equidefinable with a permutation from $\mathcal G'_1$ (see Lemma~\ref{l.Pi-1-2}).

\end{remark}

Given a word $w \in \Wa$, define $l(w) \in \omega\setminus\{0\}$, $\langle g^w_i : i \leq l(w)\rangle$ with each $g^w_i \in W_A$ (i.e., $a$ or $a^{-1}$ does not occur) and $\langle j^w_i : i < l(w)\rangle$ with each $j^w_i \in \Z\setminus\{0\}$ such that 
\begin{equation}\label{e.word.split}
w = g^w_{l(w)} a^{j^w_{l(w)-1}} g^w_{l(w)-1} \hdots a^{j^w_0} g^w_0
\end{equation}
where each $g^w_i$ is of maximal possible length, but not equal to the empty word when $0 < i <l(w)$. In other words, $g^w_{l(w)}$ and $g^w_0$ are the left-most and right-most segments of $w$ in which neither $a$ or $a^{-1}$ occurs (either can be equal to $\emptyset$) and $g^w_1, \hdots, g^w_{l-1}$ enumerate the maximal subwords, if any, of $w$ in which neither $a$ nor $a^{-1}$ occurs.

We will also write $g^w_L$ for $g^w_{l(w)}$ and $g^w_R$ for $g^w_0$ (the subscripts stand for ``left-most'' and ``right-most'', of course).

Similarly, we write 
\begin{align*}
i^w_R &=\begin{cases} 1 &\text{if $j^w_0 > 0$,}\\
-1 &\text{if $j^w_0 < 0$;}
\end{cases} \\
i^w_L &=\begin{cases} 1 &\text{if $j^w_{l(w)-1} > 0$,}\\
-1 &\text{if $j^w_{l(w)-1} < 0$.}\\
\end{cases} 
\end{align*}
\begin{example}
Supposing $w=b_1 b_0 a^{-1}a^{-1}$ with $b_0,b_1 \in A$, we have $g^w_L=b_1 b_0$, $g^w_R=\emptyset$, $l(w)=1$, $j^w_0=-2$, and $i^w_R=i^w_L=-1$.
\end{example}

\subsection{The Conjugated Subwords Lemma}

The next lemma will play a role in density arguments in Lemma~\ref{lemma.bijective} and Lemma~\ref{lemma.generic.coding}. Also, based on this lemma, restricting to words in $\Wncs$ allows a simpler definition of the forcing (see Remark~\ref{r.order}) in comparison, e.g., to its ancestor in \cite{Zhang}.
As before, the reader may think of $s$ below as a finite approximation to the generic permutation $\sigma^G$ we wish to add.
\begin{lemma}\label{l.csw}
Suppose $w \in \Wncs$, $s$ is a partial injective map from $\omega$ to $\omega$, and $n\notin\dom(s)$.
Then for all but a finite set $E^w_{s,n}$ of $n' \in\omega$, letting $s'=s\cup\{(n,n')\}$ 
it holds that $s'$ is injective and $\fix(w[s])=\fix(w[s'])$.
\end{lemma}
Note that the lemma can be read as a sufficient condition for having a conjugated subword:
Given $w$, $s$, and $n$ as in the lemma, if there are infinitely many $n'$ such that $\fix(w[s])\neq\fix(w[s'])$ for $s'=s\cup\{(n,n')\}$, then $w$ must have a proper conjugated subword.

\begin{proof}
Let $W^*$ be the set of subwords of circular shifts of $w$ and let 
\begin{equation}
\begin{split}\label{e.domain.ext.cont}
E^w_{s,n}= &\bigcup \big\{ \fix(w'[s]) \colon w' \in W^*\setminus\{ \emptyset\}\big\} \cup\\
&\big\{ w'[s]^i (n) \colon i \in \{-1,1\}, w' \in W^* \big\}\cup\\
&\ran(s).
\end{split}
\end{equation}
Let $n' \notin E^w_{s,n}$ be aritrary.
As $n'\notin\ran(s)$, $s'$ is injective and as $n'$ is not from the set in the second line of \eqref{e.domain.ext.cont}, $n' \neq n$ (noting $\emptyset \in W^*$ and $\emptyset[s]=\rho(\emptyset)$). 

Assume towards a contradiction that $m_0 \in \fix (w[s']) \setminus \fix(w[s])$. 
We will reach a contradiction by finding a proper conjugated subword (namely, $w_1$ below) of $w$.

As the $(w,s)$-path of $m_0$ differs from the $(w,s')$-path, the latter must contain an application of $a$ to $n$ or of $a^{-1}$ to $n'$. Write this latter path (omitting some steps) as
\begin{equation}\label{e.path.basic}
m_0 \stackrel{w_{l+1}}{\longleftarrow} m_{k(l)+1} \stackrel{a^{j(l)}}{\longleftarrow} m_{k(l)} 
\stackrel{w_{l}}{\longleftarrow} \hdots
\stackrel{w_1}{\longleftarrow} m_{k(0)+1} \stackrel{a^{j(0)}}{\longleftarrow} m_{k(0)} 
\stackrel{w_{0}}{\longleftarrow} m_0
\end{equation}
where for each $i \leq l$, $j(i) \in \{ -1,1\}$ and $\{ k(i) \colon i \leq l\}$ is the increasing enumeration of the set of $k$ such that
$m_k =n$ and $a$ is applied or $m_k = n'$ and $a^{-1}$ is applied at step $k$.
Thus by definition $w_i[s] = w_i [s']$ for each $i \leq l +1$.

The following hold by definition of the $k(i)$ and by choice of $n'$:
\begin{enumerate}[(i)]
\item\label{i.1} Unless $i\in\{0,l+1\}$ (i.e., $w_i$ is the first or last word in \eqref{e.path.basic}) it must hold that $w_i \neq \emptyset$:

Towards a contradiction, assume $w_i = \emptyset$ and $0<i<l+1$.
Then $j(i-1) \neq j(i)$ as $n \neq n'$.
But if
$\lh(w) \nmid \lh(a^{j(i-1)} \hdots w_0)$
then $a^{j(i)}a^{j(i-1)}$ is a subword of $w$ and 
 $j(i-1) \neq j(i)$ is impossible as adjacent $a$ and $a^{-1}$ are not allowed in the reduced word $w$.
If on the other hand
$\lh(w) \mid \lh(a^{j(i-1)} \hdots w_0)$,
then $a^{j(i-1)}$ is a left- and $a^{j(i)}$ a right-initial segment of $w$, so as $j(i-1) \neq j(i)$, $w$ has a proper conjugate subword; contradiction.
\item\label{i.2} For no $i \leq l$ is it the case that $w_{i}[s]$ sends $n$ to $n'$ or vice versa.
This is by choice of $n'$.
\item\label{i.3} For no $i \leq l$ is $n'$ a fixed point of $w_i[s]$ unless $w_i = \emptyset$, again by choice of $n'$.
\end{enumerate}
From this it follows that unless $i=0$ or $i=l+1$,
the values in the path appearing adjacent to $w_i$, i.e., $m_{k(i-1)+1}$ and $m_k(i)$, are both $n$. 
There is at least one such $i$, for the path cannot have the following form:
\begin{equation}\label{e.path.not}
m_0 
\stackrel{w_1}{\longleftarrow} m_{k(0)+1} \stackrel{a^{j(0)}}{\longleftarrow} m_{k(0)} 
\stackrel{w_{0}}{\longleftarrow} m_0
\end{equation}
for then $w_0 w_1$---a subword of a cyclic shift of $w$--- or its inverse sends $n$ to $n'$ which is impossible by choice of $n'$.
Thus $w_2$ and $j(1)$ are defined; by \ref{i.2} and \ref{i.3} we have 
$m_{k(1)} = m_{k(1)+1} = n$, $j(0) = -1$,  $j(1)=1$, and so $n \in \fix(w_1)$. 

Finally $j(2)$ cannot be defined as otherwise by \ref{i.1} $w_2$ must be non-empty and send $n'$ to one of $\{ n, n'\}$, contradicting Items \ref{i.2} or \ref{i.3} above.
So the path in \eqref{e.path.basic} has the following form:
\begin{equation}\label{e.path.semifinal}
m_0 
\stackrel{w_2}{\longleftarrow} n' \stackrel{a^{-1}}{\longleftarrow} n 
\stackrel{w_1}{\longleftarrow} n \stackrel{a}{\longleftarrow} n' 
\stackrel{w_0}{\longleftarrow} m_0
\end{equation}
As $w_0 w_2$ is a subword of a cyclic shift of $w$,
$w_0 w_2 = \emptyset$ since we made sure $n' \notin \fix(w_0 w_2[s])$ otherwise.
So $w_2 = {w_0}^{-1}$ and $w_1$ is a proper conjugate subword of $w$.
Again, we reach a contradiction.
\end{proof}

\subsection{Two Lemmas on free groups}

In the next section, we discuss computable coding, i.e., the mechanism by which each $w[\sigma^G]$ should code its assigned real, $z^w$.
For this (in particular for the proof of Lemma~\ref{l.computable.coding}) we need a combinatorial tool, Lemma~\ref{l.combinatorial} below, which is best phrased in the language of free groups. 
For reasons of clarity, we start by proving a simple special case.
\begin{lemma}\label{l.combinatorial.simple}
Let $X$ be an arbitary set. Suppose we are given $w_0, w_1 \in \operatorname{\mathbb{F}}(X)$,
both without proper conjugate subwords,
and $n_0,n_1 \in \omega$ such that 
\begin{equation}\label{e.powers.agree}
(w_0)^{n_0} = (w_1)^{n_1}.
\end{equation}
Then there exists $v \in \operatorname{\mathbb{F}}(X)$ and $m_0, m_1 \in \omega$ such that
$w_0=v^{m_0}$ and $w_1=v^{m_1}$, i.e., $w_0$ and $w_1$ have a common root.
\end{lemma}
In fact, the assumption that $w_0,w_1\in\Wncs$ can be dropped. 
The proof of this latter fact is of no matter to us and we omit it.
\begin{proof}
We may assume $n_0,n_1 > 1$.
For each $i\in\{0,1\}$, write $\lh(w_i) = c \cdot l_i$ where 
\[
c=\operatorname{gcd}\big(\lh(w_0),\lh(w_1)\big)
\]
and define, for each $k<l_i\cdot n_i$, 
\begin{gather*}
y^i_k \colon \{0,\hdots,c-1\}\to X\cup X^{-1},\\
y^i_k(j)=\text{$(k\cdot c + j)$th letter of $(w_i)^{n_i}$ (from the right)}
\end{gather*}
for each $j<c$.
In other words,
\begin{multline*}
(w_0)^{n_0} = y^0_{n_0 \cdot l_0-1} \hdots y^0_0 = \left(y^0_{l_0-1} \hdots y^0_0\right)^{n_0}=\\
(w_1)^{n_1} = y^1_{n_1 \cdot l_1-1} \hdots y^1_0 = \left(y^1_{l_1-1} \hdots y^1_0\right)^{n_1}
\end{multline*}
and so it holds for every $i\in\{0,1\}$ and every $k<l_i$ that
\begin{equation*}%
y^{1-i}_{k} = y^{1-i}_{k+l_1} = y^i_{k + l_1} = y^i_{(k+l_1) \bmod{l_i}}.
\end{equation*}
Since $l_0$ and $l_1$ are relatively prime, $k \mapsto (k+l_{1-i}) \bmod{l_i}$ defines a bijection of $l_i$ whose action (on $l_i$) has a single orbit. 
Note that in this calculation, no $k \geq l_0+l_1$ occurs.
Letting $v = y^0_0$, it follows that for every $i \in \{0,1\}$ and every $k<l_i$, $y^i_k = v$, proving the lemma. 
\end{proof}
We can see from the proof that the conclusion of Lemma~\ref{l.combinatorial.simple} holds also under weaker assumptions. 
It is enough, e.g., that the first $\lh(w_0)+\lh(w_1)$ letters of the two words in Equation \eqref{e.powers.agree} agree.
In fact, we are interested in a situation equivalent to the one just described.
Let us therefore prove the following stronger form of the previous, taylored precisely to the situation we shall find ourselves in (see Claim~\ref{e.g_L.differ} below):
\begin{lemma}\label{l.combinatorial}
Let $X$ be an arbitary set and suppose we are given $w'_0, w'_1 \in \operatorname{\mathbb{F}}(X)$
and $n_0,n_1 \in \omega$.
For each $i\in\{0,1\}$, write $l'_i = \lh\big((w'_i)^{n_i}\big)$ and let
\[
(w'_i)^{n_i} = w_{i,l'_i-1}  \hdots w_{i,0} 
\]
in reduced form, with $w_{i,j}\in X\cup X^{-1}$ for each $j<l'_i$.
Suppose further we have 
two intervalls $[a_i, b_i] \subseteq \{0, \hdots, l_i-1\}$ (for $i\in\{0,1\}$) of equal length such that letting $f \colon [a_0,b_0]\to[a_1,b_1]$ be the unique order preserving bijection, it holds that
\[
(\forall j \in [a_0,b_0])\; w_{0,j}=w_{1,f(j)}.
\]
Then provided that 
\begin{equation}\label{e.few}
b_0 - a_0 \geq \lh(w'_0)+\lh(w'_1)
\end{equation}
there exists $v_0,v_1 \in \operatorname{\mathbb{F}}(X)$ and $m_0, m_1 \in \omega$ such that
$w'_0=(v_0)^{m_0}$ and $w'_1=(v_1)^{m_1}$ and $v_1$ is a cyclic shift of $v_0$ (with offset $a_1 - a_0$).
\end{lemma}
The statement of the lemma is much more involved than its proof.
\begin{proof}
Let us first simplify by arranging that $a_0=a_1=0$; this is possible by replacing $w'_i$, for each $i\in \{0,1\}$, by its cyclic shift $w_i$ with offset $a_i$.
By \eqref{e.few}, $f$ is just the identity on a superset of $\{0, \hdots, \lh(w_0)+\lh(w_1)-1\}$ and 
we find ourselve almost in the situation of the previous lemma, except that we only have equality on the first $\lh(w_0)+\lh(w_1)$ letters.
By the proof of the previous lemma, we conclude that $w_0$ and $w_1$ have a common root $v$.

Now observe that for each $i\in\{0,1\}$, from the root $v$ of the cyclic shift $w_i$ of $w'_i$, we obtain a root $v_i$ of $w'_i$ by taking a cyclic shift:
Let $v_i$ be the cyclic shift of $v$ with offset $-a_i$.
\end{proof}

\subsection{Computable Coding}\label{s.computable.coding}

As the final prerequisite to defining the forcing $\QQ$ we need to fix the algorithm used in the computable coding of $z^{w}$ by the permutation $w[\sigma^G]$ (as discussed at the beginning of this section).

\begin{definition}[Computable coding, Part 1]\label{d.computable.coding.1}~
\begin{enumerate}
\item\label{d.partition} Fix a computable sequence $\la P_n : n\in \omega\ra$ of infinite sets $P_n\subseteq \omega$; for concreteness, let us say $P_n$ consists of the natural numbers which are divisible $2^n$ but not by $2^{n+1}$.

\item\label{i.sufficiently.generic} Let us say that our cofinitary representation\footnote{For the definition of cofinitary representation, see p.~\pageref{d.cofinitary.representation}.} $\rho$ is \emph{sufficiently generic} to mean that
for any finite sequence $\bar n \in \omega^{l+1}$ and pairwise distinct $g_1, \hdots, g_{l}\in W_A$,
the set of $m$ such that 
\begin{align*}
m &\in P_{\bar n(0)},\\
\rho(g_1)(m) &\in P_{\bar n(1)},\\
\vdots\\
\rho(g_{l})(m) &\in P_{\bar n(l)},
\end{align*}
is infinite.
\end{enumerate}
\end{definition}
\begin{assumption}\label{a.sufficiently.generic}
Let us assume throughout this section that $\rho : A \to S_\infty$ is a sufficiently generic cofinitary representation.
\end{assumption}
It will simplify coding substantially to consider only words without proper conjugated subwords or roots. 
This will the first step of the restriction described in Remark~\ref{r.subwords}, and as has been explained there, it will not endanger our goal of adjoining a $\Pi^1_2$ MCG. Compare also Remark~\ref{r.nr} below.
\begin{definition}[Computable coding, Part 2]\label{d.computable.coding.2}~
Let a sequence $\bar n \in \omega^{\leq\omega}$ be given. 
Suppose $\sigma$ is a partial function from $\omega$ to $\omega$, and $w\in \Wncs\cap\Wnr$. 

\begin{enumerate}%
\item
We say $(w,\sigma)$ \emph{precodes} $\bar n$ \emph{with parameter} $m$ if and only if  
\begin{equation}\label{e.code}
(\forall k < \dom(\bar n)) \; w^{3(k+1)}[\sigma](m) \in P_{\bar n(k)}.
\end{equation}
The factor $3$ in the exponent helps to ensure we have sufficient opportunity to allow paths for different words to diverge, in the density argument showing that it is forced that $w[\sigma^G]$ codes $z^w$ (see Lemmas~\ref{l.computable.coding} and \ref{lemma.generic.coding}).

\item\label{i.exactly} Suppose now that $\dom(\bar n) = k$. 
We say that  $(w,\sigma)$  \emph{exactly precodes} $\bar n$
\emph{with parameter} $m$ if $k <\omega$, $(w,\sigma)$ precodes $\bar n$ 
and in addition
\[
a^{i^w_R} g^w_R w^{3k}[\sigma](m)\text{  is undefined,}
\]
recalling that $i^w_R$ is the exponent of the right-most occurrence of $a$ or $a^{-1}$ in $w$.
In other words, $(w,\sigma)$ exactly precodes $\bar n$ if the path of $m$ under $(w,\sigma)$ is of minimal possible length under the requirement that it precodes $\bar n$. 
Note that the precise form of the above definition depends on our assumption that $w\in \Wncs$.
Like the factor $3$ above, this notion will help separate paths for different words in the density argument showing that it is forced that $w[\sigma^G]$ codes $z^w$.

\item\label{i.critical} We say that $m'$ is \emph{the critical point in the path of $m$ under $(w,\sigma)$} if for some $k\in\omega$,
\[
m' = (g^w_L a^{i^w_L})^{-1}w^{3(k+1)}[\sigma](m) 
\]
recalling that $i^w_L$ is the exponent of the left-most occurrence of $a$ or $a^{-1}$ in $w$.
We define this terminology because, when extending $\sigma$ so that the path of $m$ increases in length with the purpose of achieving exact precoding of a given $\bar n$, it is precisely at critical points that exact precoding imposes a non-trivial requirement for this extension. 

\item We say that $m'$ is a \emph{coding point} in the path of $m$ under $(w,\sigma)$ if 
$m' =  w^{3k}[\sigma](m)$ for some $k\in\omega$ (although for $k=0$, nothing is coded; we take the view that the trivial, i.e., empty sequence is precoded in this case).

\item\label{d.S} Fix a bijection $S\colon \omega \to 2^{<\omega}$.
Suppose $z \subseteq \omega$ and $l\in \omega+1$. We say \emph{$(w,s)$ codes (resp., exactly codes) $z$ up to $l$ with parameter $m$} if and only if there is a sequence $\bar n$ such that $(w,s)$ precodes (resp., exactly precodes) $\bar n$ with parameter $m$ and
\[
\chi_z \res l = \bigcup\big\{S\big(\bar n(i)\big) : i<\dom(\bar n)\big\},
\]
where $\chi_z$ denotes the characteristic function of $z$.
We just say \emph{$(w,s)$ codes $z$ (with parameter $m$)} if $(w,s)$ codes $z$ up to $\omega$ (with $m$; if we don't mention $m$, the phrase is understood to mean ``with some parameter $m$'').
\end{enumerate}
\end{definition}
\begin{remark}\label{r.nr}
We restrict attention to words without proper conjugate subwords (see Notation~\ref{n.words}) to simplify the
form of \eqref{i.exactly} and \eqref{i.critical} in the above definition, as well as the argument of Lemma~\ref{l.computable.coding} below.
We moreover restrict to words \emph{without roots} (again, see Notation~\ref{n.words}).
This avoids conflicting coding requirements arising from the demands that $w[\sigma^G]$ code $z^w$ and at the same time, $w^m[\sigma^G]$ code $z^{w^m}$ for some $m>1$; such conflicts are avoided by never demanding the latter.
This is crucial in the proof Lemma~\ref{l.computable.coding} below.
As has been mentioned, neither restriction thwarts our goal of constructing a $\Pi^1_2$ MCG (see  Lemma~\ref{l.Pi-1-2} and Remark~\ref{r.subwords}).
\end{remark}

Next, we prove the crucial lemma for the computable coding of $z^w$ by $w[\sigma^G]$.
This lemma is connected to the fact that given any finite (or even countable) partial assignment $w \mapsto z^w$ from $\Wnr\cap\Wncs$ to $\powerset(\omega)$, the set of permutations $\sigma\in S_\infty$ such that $w[\sigma]$ codes $z^w$ whenever $z^w$ is defined, %
is comeager.
The lemma will be the basis of the density argument in Lemma~\ref{lemma.generic.coding} below.

\begin{lemma}\label{l.computable.coding}
Suppose we have a finite set of words $W\subseteq \Wnr\cap\Wncs$, and assignments $w \mapsto z^w \in \powerset(\omega)$ and $w \mapsto m^w \in \omega$  for $w\in W$. 
Suppose further we have a finite, partial, and injective map $s \colon \omega \rightharpoonup \omega$ 
so that for each $w\in W$, $(w,s)$ exactly codes $z^w$ up to some finite length with parameter $m^w$.

Then given any $l' \in \omega$, we may find a finite, partial, and injective map $s' \colon \omega \rightharpoonup \omega$ such that $s \subseteq s'$ and for each $w\in W$, $(w,s')$ exactly codes $z^w$ up to some (finite) length $l$ with $l\geq l'$.
\end{lemma}
\begin{proof}
Let $s$, $W$, $w \mapsto z^w \in \powerset(\omega)$ and $w \mapsto m^w \in \omega$ for $w\in W$, and $l'\in\omega$ be given as in the lemma.
For $w\in W$, call the path of $m^w$ under $(w,s)$ the \emph{coding path} for $(w,s)$. 
Write $m^w_{\textup{tm}}$ for the terminal value in the coding path for $(w,s)$
and let 
$M=\{m^w_{\textup{tm}} : w \in W\}$.
Fix some $m^*_0 \in M$.
It suffices to find $s'$ such that for each $w \in W$,
$(w,s')$ exactly codes $z^w \res l'$ if $m^w_{\textup{tm}} = m^*_0$, and the coding path of $(w,s')$ equals the coding path of $(w,s)$ (and thus $(w,s)$ still exactly codes an initial segment of $z^l$) 
if $m^w_{\textup{tm}} \neq m^*_0$.
The argument may then be repeated for all elements of $M$.

For each $w\in W$ with $m^w_{\textup{tm}} = m^*_0$, define
\[
w^*(w) = (a^{i^w_L} g^w_L)^{-1}w^3 (g^w_R)^{-1},
\]
that is, $w^*(w)$ is the word whose intepretation under $s'$, once defined, will take $m^*_0$ to the next critical point in the coding path of $(w,s')$.

Consider the set
\[
T =\big\{w \in W_{A,a} : (\exists w' \in W)\;m^w_{\textup{tm}} = m^*_0 \land w^*(w') \sqsupseteq w\},
\]
which forms a tree (with root $\emptyset$) under $\sqsupseteq$ (defined on p.~\pageref{d.sqsupseteq}).
We will define $s_0, s_1, \hdots, s_k$ by finite induction, with $s_0=s$, and so that $s_i \subseteq s_{i+1}$ for each $i<k$; $s'$ will be $s_k$. 
In the course of our finite induction, we will deal with all words from $T$, shorter words first.
So let $\la w_i : i<k\ra$ enumerate $T$ so that $i<j \Rightarrow \lh(w_i) \leq \lh(w_j)$.

We shall take care to ensure the following properties in our induction: 
\begin{induction}\label{IH}~
\begin{enumerate}[(P1)]
\item\label{P.injective}\label{P.first}
If $w',w'' \in T$ are distinct and 
$m^*_0 \in \dom(w'[s_i])\cap \dom(w''[s_i])$, 
then 
\[
w'[s_i](m^*_0) \neq w''[s_i](m^*_0);
\]
\item\label{P.m^*_0} For any $w \in W$ such that $m^{w}_{\textup{tm}} \neq m^*_0$ 
and any $w' \in T\setminus\{\emptyset\}$ such that 
$m^*_0 \in \dom(w'[s_i])$
\[
w'[s_i](m^*_0) \notin \set(w,s_i,m^w).
\]
\item\label{P.path.identical.W}\label{P.last} For any $w \in W$ such that $m^{w}_{\textup{tm}} \neq m^*_0$, 
the coding path for $(w,s_i)$ remains identical to the coding path for $(w,s_0)$.
\end{enumerate}
\end{induction}
For $i=0$, \ref{P.injective}--\ref{P.last} already hold because then, $s_0=s$ and the only $w\in T$ such that $m^*_0 \in \dom(w[s_i])$ is the empty words, $w=\emptyset$.

For the inductive step, suppose we already have constructed $s_i$ and that \ref{P.injective}--\ref{P.last} hold. 
If $m^*_0 \in \dom(w_i[s_i])$, we can let $s_{i+1}=s_i$, so assume otherwise.
We can assume by induction that $w_i \in \{aw, a^{-1}w\}$ with $m^*_0 \in \dom(w[s_i])$.
We shall assume for simplicity that 
\[
w_i = aw,
\] 
and leave the case $w_i=a^{-1}w$ to the reader.
Let
\[
m=w [s_i]\big(m^*_0).
\]
We will define $s_{i+1}=s_i\{(m,m')\}$, where $m'$ is an arbitrary element of $\omega$ satisfying the following requirements.

\begin{enumerate}[(R1)]
\item\label{r.comp.first} For each $g \in \mathbb{F}(A)$ (including $g=\emptyset$) such that $gaw \in T$, and for each $w' \in W$, 
\[
\rho(g)(m') \notin \set(m^{w'},w',s_i) \cup\dom(s)\cup\ran(s).
\]
\item\label{r.comp.g} For any $g_0,g_1 \in \mathbb{F}(A)$ such that $g_0aw \in T$, $g_1 aw \in T$, and $g_0\neq g_1$ it holds that $g_0(m')\neq g_1(m')$.

\item\label{r.comp.last} For any $w' \in W$  such that
$m$ is a critical point in the coding path for $(w',s_i)$ and $a^{i^{w}_L}=a$
it holds that
\[
\rho\big(g^{w'}_L\big)(m') \in P_{S(\chi \res l')}
\]
where $\chi$ is the characteristic function of $z^w$ (see Definition~\ref{d.computable.coding.1}\eqref{d.partition},  p.~\pageref{d.partition} for $P_n$ and Definition~\ref{d.computable.coding.2}\eqref{d.S}, p.~\pageref{d.S} for $S$).
\end{enumerate}
\begin{claim}\label{c.inf.m'}
There are infinitely many $m'\in\omega$ sarisfying \ref{r.comp.first}--\ref{r.comp.last} above. 
\end{claim}
Let us postpone the proof of the claim and first show how the lemma follows, asuming the claim.
Then since we may find $m'$ satisfying \ref{r.comp.first}--\ref{r.comp.last}, we have succeded in defining $s_{i+1}$.

Note that the coding path of $(w_i,s_{i+1})$ is longer by at most one application of $a$ or $a^{-1}$ (but not both) than the coding path of $(w_i,s_i)$. Moreover the coding path of any $(w_j,s_{i+1})$ is either the same as that of $(w_i,s_{i+1})$, this being the case when $w_j\sqsupseteq w_i$; or otherwise it is identical to the coding path of $(w_j,s_i)$, i.e., extending $s_i$ to $s_{i+1}$ leaves this path unchanged.
The properties \ref{P.first}--\ref{P.last} listed in~\ref{IH} are preserved due to \ref{r.comp.first} and \ref{r.comp.g}. 

This finishes the inductive step in our construction of $s_0, \hdots, s_k$, and thus, of $s' =s_k$.
By \ref{r.comp.last} and by construction, $(w,s')$ exactly codes $z^w$ up to $l'$ for each $w\in W$ with $m^w_{\textup{tm}}= m^*_0$.
By \ref{P.path.identical.W} the coding path of $(w,s')$ equals the coding path of $(w,s)$ for each $w\in W$ with $m^w_{\textup{tm}}\neq m^*_0$.

\medskip

To finish the proof of Lemma~\ref{l.computable.coding}, it remains to prove Claim~\ref{c.inf.m'}.
\begin{proof}[Proof of Claim~\ref{c.inf.m'}]
It is clear that \ref{r.comp.first} and \ref{r.comp.g} only exclude finitely many $m'$ (in the case of \ref{r.comp.g} we use that $\rho$ is a cofinitary representation). 
It remains to show that there are infinitely many $m'$ satisfying \ref{r.comp.last}.
The problem is to show that when $m$ is critical in the path for several words from $W$, none the requirements are in conflict.

To see this we need another claim:

\begin{claim}\label{e.g_L.differ}
If $w', w'' \in W$ are distinct and $m$ is a critical point in the coding paths both for $(w',s_i)$ and $(w'',s_i)$, then $g^{w'}_L \neq g^{w''}_L$. 
\end{claim}
\begin{proof}[Proof of Claim~\ref{e.g_L.differ}]
Fix $w'$ and $w''$ as in the claim. 
Since by assumption both $(w',s)$ and $(w'',s)$ also exactly precode some finite sequence, 
the last coding point in the coding path of $(w',s)$ (resp., of $(w'',s)$) must be $(g^{w'}_R)^{-1} (m^{w'}_{\textup{tm}})$ (resp., $(g^{w''}_R)^{-1} (m^{w''}_{\textup{tm}})$). 
By definition of exact coding and of critical points, for some $k',k'' \in \omega\setminus\{0\}$ we have
\begin{align*}
(g^{w'}_L a^{i^{w'}_L})^{-1}(w')^{3k'} (g^{w'}_R)^{-1}[s_i] (m^{w'}_{\textup{tm}})=&m,\\
(g^{w''}_L a^{i^{w''}_L})^{-1}(w'')^{3k''}(g^{w''}_R)^{-1}[s_i] (m^{w''}_{\textup{tm}})=&m.
\end{align*}
By \ref{P.injective}--\ref{P.path.identical.W}, this is only possible if $m^{w'}_{\textup{tm}} = m^{w'}_{\textup{tm}} = m^*_0$ and both words on the left of the above equations are equal to $w_i$ 
and therefore 
\begin{equation}\label{e.words.equal}
(g^{w'}_L a^{i^{w'}_L})^{-1}(w')^{3k'} (g^{w'}_R)^{-1}= (g^{w''}_L a^{i^{w''}_L})^{-1}(w'')^{3k''}(g^{w''}_R)^{-1}.
\end{equation}
By \eqref{e.words.equal}, $(w')^{3k'}$ and $(w'')^{3k''}$ agree as in the hypothesis of Lemma~\ref{l.combinatorial} on an interval $[a_0,b_0]$ of length 
at least $2\cdot\max\big(\lh(w'),\lh(w'')\big)$. %
It follows by Lemma~\ref{l.combinatorial} that there are $v', v''$ and $\tilde k', \tilde k''$ such that
$w'= (v')^{\tilde k'}$ and $w''= (v'')^{\tilde k''}$ and %
$v''$ is a cyclic shift of $v'$ with an offset of $\lh(g^{w'}_R) - \lh(g^{w''}_R)$.
It also follows that $(w')^{k'}$ is a cyclic shift of $(w'')^{k''}$ by the same offset,
i.e., $i^{w'}_L=i^{w''}_L$ and
\begin{equation*}%
g^{w'}_R g^{w'}_L = g^{w''}_R g^{w''}_L.
\end{equation*}
If $\lh(w') \neq \lh(w'')$, it follows from $\lh(v')=\operatorname{gcd}\big(\lh(w'),\lh(w'')\big)$ that  then $v'$ is a \emph{proper} subword of $w'$ or $v''$ is a \emph{proper} subword of $w''$, in contradiction to our assumption that $w',w'' \notin\Wnr$.
Thus, $\lh(w') = \lh(w'')$ and $w'$ is a cyclic shift of $w''$
with an offset of $\lh(g^{w'}_R) - \lh(g^{w''}_R)$.
As $w' \neq w''$ by assumption, we conclude $g^{w'}_L \neq g^{w''}_L$.
\renewcommand{\qedsymbol}{{\tiny Claim~\ref{e.g_L.differ}.} $\Box$}
\end{proof} 
Using this second claim, we are able to finish the proof of Claim~\ref{c.inf.m'}:
Writing $W^m$ for the set of $w\in W$ such that
$m^w_{\textup{tm}} = m^*_0$
and
$m$ is a critical point in the coding path for $(w',s_i)$, by Claim~\ref{e.g_L.differ} the map 
\begin{gather*}
W^m \to W_A,\\
w' \mapsto g^{w'}_L
\end{gather*}
is injective.
Since by Assumption~\ref{a.sufficiently.generic}, $\rho$ is sufficiently generic, there are infinitely many $m'$ satisfying
\[
\rho\big(g^{w'}_L(m')\big) \in P_{S(z^{w'}\cap l')}.
\]
for each $w' \in W^m$. Hence there are infinitely many $m'$ satisfying \ref{r.comp.last}. 
\renewcommand{\qedsymbol}{{\tiny Claim~\ref{c.inf.m'}.} $\Box$}
\end{proof} 
As we have argued above just after the statement of Claim~\ref{c.inf.m'}, this proves the lemma.
\renewcommand{\qedsymbol}{{\tiny Lemma~\ref{l.computable.coding}.} $\Box$}
\end{proof}

\subsection{The forcing}

It is high time we define the forcing
$\QQ$, as promised at the beginning of this section.
To this end let us suppose that we are given
\begin{itemize}

\item A sufficiently generic cofinitary representation $\rho\colon A \to S_\infty$ (see the beginning of Section~\ref{s.finite.partial} and Definition~\ref{d.computable.coding.1}\eqref{i.sufficiently.generic})
and a ``new'' index $a$, i.e., $\{a,a^{-1}\}\cap A=\emptyset$,

\item $\mathcal{F}=\{f_{m,\xi}:m\in\omega,\xi\in\omega_1\}$, an %
almost disjoint family of permutations (i.e., the graphs are pairwise almost disjoint subsets of $\omega\times\omega$)
so that $ f_{m,\xi} \notin \la\im(\rho)\ra$ and $\la \im(\rho), f_{m,\xi}\ra$ is cofinitary for each $m\in\omega,\xi\in\omega_1$.

\item For each $w \in \Wncs$, a family  $\mathcal{Y}^w=\{Y^w_m:m\in\omega\}$ of subsets of $\omega_1$, 

\item For each $w \in \Wncs$, a subset $z^w$ of $\omega$.
 
\end{itemize}
Write $\mathcal Y$ for $\la \mathcal{Y}^w : w \in \Wncs\ra$, and $\bar z$ for $\la z^w : w \in \Wncs\ra$.
(We will not give concrete values to $\rho$, $\mathcal F$, $\mathcal Y$, and $\bar z$ until Section~\ref{s.iteration}; see in particular the roadmap at the beginning of said section, and Section~\ref{subsection.group}. See the beginning of Section~\ref{section.group.forcing} for an informal discussion of $\rho$, $\mathcal F$, $\mathcal Y$, and $\bar z$.)

We define a forcing, denoted $\QQ^{\mathcal{F},\mathcal{Y},\bar z}_{\rho,\{a\}}$, which adjoins a generic permutation $\sigma^G$ such that the mapping from $A\cup\{a\}$ to $S_\infty$ which
extends $\rho$ and sends $a$ to $\sigma^G$ 
is a cofinitary representation; moreover, 
\begin{itemize}
\item for each $w \in \Wncs\cap \Wnr$  the permutation $w[\sigma^G]$ codes (in the sense of Definition~\ref{d.computable.coding.2}) the real $z^w$,
\item for each $w \in \Wncs$ and $m\in \psi[\sigma^G]$, $w[\sigma^G]$ almost disjointly via the family $\mathcal{F}^m=\{f_{m,\xi}: \xi\in\omega_1\}$ codes $Y^w_m$.
\end{itemize}

A condition in $\QQ=\QQ^{\mathcal{F},\mathcal{Y}, \bar z}_{\rho,\{a\}}$ 
is a tuple $p=\la s^p, F^p, \bar m^p, s^{p,*}\ra$ such that: 

\begin{enumerate}[label=(\Alph*),ref=\Alph*]
\item $s^p$ is an injective finite partial function from $\omega$ to $\omega$;
\item $F^p$ is a finite subset of $\Wncs$;
\item $\bar m^p = \la m^p_w:w\in \dom(\bar m^p)\ra$ where $\dom(\bar m^p)$ is a finite subset of $\Wncs\cap\Wnr$ and $m^p_w\in \omega$ for each $w\in \dom(\bar m^p)$;
\item $s^{p,*} = \la s^{p,*}_w:w\in \dom(s^{p,*})\ra$ where $\dom(s^{p,*})$ is a finite subset of $\Wncs\cap\Wnr$ and for each $w$
$$s^{p,*}_w\in [\{f_{m,\xi}: m\in \psi[w[s]],\xi\in Y^w_m\}]^{<\omega};$$

\item\label{Q.exact} For each $w \in \dom(\bar m^p)$ there exists a (unique) $l$ which we denote by $l^p_w$ such that $(w,s)$ exactly codes $z^w$ up to $l$ with parameter $m^p_w$.
\end{enumerate}
The extension relation for $\QQ$ is defined as follows: For
$q=\la s^q,F^q,\bar{m}^q, s^{q,*}\ra$ and $p=\la s^p,F^p,\bar{m}^p, s^{p,*}\ra$, we let $q \leq p$ if and only if all of the following hold:
\begin{enumerate}[(a)]
\item $s^q$ end-extends $s^p$, $F^q\supseteq F^p$,
\item\label{Q.order.fixpoints} For every $w\in F^p$,
$\hbox{fix}(w[s^q])=\fix(w'[s^p])$,

\item\label{Q.order.F} $s^{q,*}\supseteq s^{p,*}$, and for all $w\in\dom(s^{p,*})$ and  $f\in s^{p,*}_w$ we have 
$$\big(w[s^q]\backslash w[s^p]\big)\cap f=\emptyset,$$
\item $\bar{m}^q\supseteq \bar{m}^p$.
\end{enumerate}

\begin{remark}\label{r.order}
Note that the definition of $\leq$ is simpler than, e.g., in \cite{Zhang}. This is made possible by 
Lemma~\ref{l.csw}.
\end{remark}

\begin{prop}\label{prop.group.forcing}
Let $G$ be a $\QQ$-generic filter and let
$$\sigma^G=\bigcup\{s:\exists F,\bar{m}, s^*\hbox{ s.t. }\la s,F,\bar{m}, s^*\ra\in G\}.$$
The permutation $\sigma^G$ has the following properties:

\begin{enumerate}[label=(\Roman*)]
\item\label{i.A} The group $\la \im(\rho )\cup\{\sigma^G \}\ra$ is cofinitary.

\item\label{i.B} If $f$ is a ground model permutation, $f\notin \la\im(\rho)\ra$,  $\la \{f\}\cup\im(\rho)\ra$ is cofinitary 
 and $f$ is not covered by finitely many permutations in $\mathcal{F}$, then there are infinitely many $n$ such that
$f(n)=\sigma^G(n)$ and so $\la\im(\rho)\cup\{\sigma^G\}\cup\{f\}\ra$ is not cofinitary;

\item\label{i.C} For each $w\in \Wncs\cap\Wnr$, 
$w[\sigma^G]$ codes $z^w$;

\item\label{i.D} For each $w\in \Wncs$, for all $m\in\psi[w[\sigma^G ]]$, and for all $\xi\in\omega_1$
$$|w[\sigma^G ]\cap f_{m,\xi}|<\omega\hbox{ iff }\xi\in Y^w_{m},$$
that is, $w[\sigma^G ]$ codes $Y^w_{m}$ using the almost disjoint family $\mathcal F^m=\{f_{m,\xi}: \xi\in\omega_1\}$. 
\item\label{p.sufficiently.generic} The cofinitary representation $\rho^G \colon A\cup\{a\}\to S_\infty$ given by $\rho^G=\rho\cup\{(a,\sigma^G)\}$ 
is sufficiently generic.
\end{enumerate}
Moreover the forcing $\QQ$ has the Knaster property.
\end{prop}

We shall now show these properties to hold, in a series of lemmas.
The first lemma serves to show Property \ref{i.C}.
\begin{lemma}[Generic Coding]\label{lemma.generic.coding}
For any $w \in \Wa$ and any $l \in \N$, let $D^{\textup{code}}_{w,l}$ denote the set of $q\in\Q$ such that $w \in \dom(\bar m^q)$ and for some $l' \geq l$,
$q$ exactly codes $z^w$ up to $l'$ with parameter $\bar m^q_w$. Then $D^{\textup{code}}_{w,l}$ is dense in $\Q$.
\end{lemma}
\begin{proof}
Suppose $p\in\QQ$ and $w\in \Wncs\cap\Wnr$ are given.
If $w \notin \dom(\bar m^p)$ 
it is clear that we can choose $m$ large enough so that letting 
\[
q = \la s^p, F^p, \bar m^p \cup\{(w,m)\},s^{p,*}\ra
\]
we obtain a condition $q \in \QQ$ with $l^q_w = 0$ (namely, chose $m$ large enough so that $(w,s^p)$ exactly precodes the trivial sequence $\bar n=\emptyset$ with parameter $m$).

So suppose $w \in \dom(\bar m^p)$. 
We find $s' \supseteq s^p$ such that letting
\[
q = \la s', F^p, \bar m^p,s^{p,*}\ra
\]
we obtain a condition $q \in \QQ$ with $l^q_w = l'$. 
To this end, construct $s'$ as in the proof of Lemma~\ref{l.computable.coding} with $W=\dom(\bar m^p)$ and the assignment given by $\bar m^p$, making the following adaptation:
Each time we extend $s_i$ to $s_{i+1} = s_i \cup\{(m,m')\}$, choose $m'$ satisfying  
\ref{r.comp.first}--\ref{r.comp.last} and in addition, require  
\[
m' \notin \bigcup_{w \in F^p} E^w_{s_i,m}, 
\]
where $E^w_{s_i,m}$ is defined as in \eqref{e.domain.ext.cont}.
(Such $m'$ exists since the set of $m'$ satisfying \ref{r.comp.first}--\ref{r.comp.last} is infinite and each $E^w_{s_i,m}$ is finite.) 
Then by the proof of Lemma~\ref{l.computable.coding}, $w[s']$ will exaclty code $z^w$ up to $l'$ for each $w\in \dom(\bar m^p)$; and by Lemma~\ref{l.csw}, $\fix(w[s'])=\fix(w[s])$ for each $w\in F^p$ and so $q\leq p$.
\end{proof}

The next lemma shows that $\sigma^G$ is a permutation of $\omega$. 
\begin{lemma}\label{lemma.bijective}
For each $n\in\omega$  the sets
$D_n=\{q\in\QQ: n\in\dom(s^q)\}$ and $D^n=\{q\in\QQ:n\in\ran(s^q)\}$
are dense in $\QQ$.
\end{lemma}

\begin{proof}
To see $D_n$ is dense, let $p\in \QQ$ be given and find $q\in D_n$, $q \leq p$.
Assume that $n\notin \dom(s^p)$. 
If $n$ occurs as the terminal value in a coding path, the previous lemma applies.
Otherwise, find 
\[
n' \notin \bigcup_{w \in F^p} E^w_{s^p,m}, 
\]
where where $E^w_{s^p,m}$ is defined as in \eqref{e.domain.ext.cont}.
Let $s'=s^p\cup\{(n,n')\}$ and $q=\la s', F^p, \bar m^p,s^{p,*}\ra$.
Then $q \in \QQ$ and $q\leq p$ by Lemma~\ref{l.csw}.
The case $D^n$ is symmetrical and is left to the reader.
\end{proof}

Property \ref{i.A} above is established by the following lemma.
\begin{lemma}
For each $w\in W^*_{\mathbb{F}(A),a}$, the set 
\[
D_w=\{q\in\QQ: q\forces\left|\fix(w[\sigma^G])\right|<\infty\}
\]
is dense in $\QQ$.
\end{lemma}
\begin{proof}
First note that $q\forces\lvert\fix(w[\sigma^G])\rvert<\infty$ if $w \in F^q$: 
This is because by the definition of the ordering on $\QQ$,
$q\forces \fix(w[\sigma^G]) = \fix (w[s^q])$.

Therefore clearly $D_w$ is dense, since we may always add the shortest conjugated subword $w$ of any word $w' \in \Wa$ to $F^q$ to form a new condition, and of course $\fix(w'[\sigma^G]) = \fix(w[\sigma^G])$.
\end{proof}
The next lemma shows Property \ref{i.B} above. Moreover,  Property \ref{i.D} is a direct corollary to this lemma and the almost disjoint requirement in the extension relation of our forcing. 
\begin{lemma}\label{l.generic.hitting}
Suppose we are given $m\in\omega$, $w\in \Wncs$ and $\tau \in S_\infty$.
\begin{enumerate}
\item If $\tau \notin \langle \im(\rho)\rangle$, $\langle \im(\rho), \tau\rangle$ is cofinitary, and $\tau$ is not covered by finitely many elements of $\mathcal{F}$, the set $D^{\textup{hit}}_{\tau,m} = \{q\in \Q\;\colon (\exists n\geq m) \; w[s^q](n)=\tau(n)\}$ is dense.

\item If $\tau \in \mathcal{F}$, $\tau = f^w_{n,\xi}$, and $\xi \notin Y^w_m$
then too is the set $D^{\textup{hit}}_{\tau,m}$ dense.

\item If $\tau \in \mathcal{F}$, $\tau = f^w_{n,\xi}$, and $\xi \in Y^w_m$
the set $D^{\textup{hit}}_{\tau,m}\cup\{ p \in \QQ : n \in \psi[w[s^p]]\}$ is dense in $\QQ$.
\end{enumerate}
\end{lemma}
\begin{proof}
Let $\tau$ and $m$ as in the lemma be given. 
Note that in all three cases $\tau \notin \langle \im(\rho)\rangle$ and $\langle \im(\rho), \tau\rangle$ is cofinitary 
 and we can assume $\tau \notin s^{*,p}$ (for in the third case, otherwise $\Psi^{-1}(n)\in w[s^p]$) and therefore that 
\begin{equation}\label{e.tau.infinite}
\lvert \tau \setminus \bigcup  s^{p,*} \rvert = \omega.
\end{equation}
Let $W^*$ be the set of subwords of circular shifts of words in $F^p$, and find $n\in\omega\setminus m$ such that 
\begin{align} 
\label{e.tau.fix} n &\notin  \tau^{-1}\Big[ \bigcup \big\{\fix(w[s^p]) \colon w \in W^*\setminus\{ \emptyset\}\big\}\Big],\\
\label{e.tau.fix2} n &\notin  \bigcup \big\{\fix(\tau^{-1} w^i[s^p])  \colon i \in \{-1,1\}, w \in W^*\big\}, \\
\label{e.tau.use} n &\notin U\cup \tau^{-1}[U]\text{, where}\\
&U=\bigcup \big\{\set(w,s^p,\bar m^p_w) : w \in \dom(\bar m^p)\big\}\Big]\text{, and}  \notag \\
\label{e.tau.f}  n &\notin \bigcup \big\{\fix(\tau^{-1}f) : f \in s^{p,*}\big\}
\end{align}
Requirements \eqref{e.tau.fix}--\eqref{e.tau.use} exclude only finitely many $n$ (noting $\tau \notin \langle \im(\rho)\rangle$ and $\langle \im(\rho), \tau\rangle$ is cofinitary).
Since \eqref{e.tau.f} holds for infinitely many $n$ by \eqref{e.tau.infinite}, we can pick $n$ satisfying \eqref{e.tau.f}--\eqref{e.tau.use}.
Now \eqref{e.tau.fix} and \eqref{e.tau.fix2} implies
\begin{equation}
(\forall w \in F^p)\; \tau(n) \notin E^w_{s^p,n} \label{e.E} 
\end{equation}
(recalling $E^w_{s^p,n}$ was defined in \eqref{e.domain.ext.cont}) and by \eqref{e.tau.use} and \eqref{e.tau.f},
\begin{gather}
(\forall w \in \dom(\bar m^p))\; \{n,\tau(n)\} \cap \set(w,s^p,\bar m^p_w) = \emptyset. \label{e.use}   \\
(\forall f \in s^{p,*})\; \tau(n) \neq f(n), \label{e.f}
\end{gather}
By \eqref{e.f}, \eqref{e.use}, and the proof of Lemma~\ref{l.computable.coding}, letting $s=s^p\cup\{(n,\tau(n))\}$ and $q = \la s, F^p, \bar m^p, s^{p,*}\ra$ we obtain $q\in\QQ$. 
By \eqref{e.E} and Lemma~\ref{l.csw}, $q\leq p$. 
Clearly, $q\in D^{\textup{hit}}_{\tau,m}$.
\end{proof}

Next, we show Property \ref{p.sufficiently.generic}.
\begin{lemma}
It is forced by $\QQ$ that 
the cofinitary representation given by $\rho\cup\{(a,\sigma^G)\}$ 
is sufficiently generic.
\end{lemma}
\begin{proof}
The proof is similar to that of Lemma~\ref{lemma.generic.coding}. 
We leave the details to the reader.
\end{proof}

Finally we show the following.
\begin{lemma}\label{l.Knaster}
The forcing $\QQ$ is Knaster.
\end{lemma}
\begin{proof}
It is straightforward to check that if $p,q\in\QQ$ are such that $s^p = s^q$ and $\bar m^p$ agrees with $\bar m^q$ on $\dom(\bar m^p) \cap \dom(\bar m^q)$ then
\[
r= \la s^p, F^p\cup F^q, \bar m^p \cup \bar m^q, s^{p,*}\cup s^{q,*}\ra
\]
is a condition in $\QQ$ and $r \leq p, q$.
Therefore $\QQ$ is Knaster by a standard $\Delta$-systems argument.
\end{proof}

\section{The forcing iteration}\label{s.iteration}

Fix $M,N$ arbitrary such that $2\leq M<N<\omega$. 
Our goal is to force that $\mathfrak{a}_g=\omega_M<\mathfrak{c}=\omega_N$ with a $\Pi^1_2$ definable witness $\mathcal G$ to $\mathfrak{a}_g$.  

Our strategy will be to adjoin (using the forcing from the previous section) a definable set of permutations $\mathcal G'\subseteq S_\infty$ such that the following hold:
\begin{itemize}
\item $\mathcal G=\la\mathcal G'\ra$ is a maximal cofinitary group.
\item For a fixed sequence $\la S_\delta:\delta < \omega_M\ra$ of stationary costationary sets which is definable in $L$, $\mathcal G'$ is definable as follows: For each $g \in S_\infty$, 
\begin{equation}\label{e.def.G'}
g \in \mathcal G' \iff  \big(\exists \gamma \in \lim(\omega_M)\big)\; \psi[g]= \{m \in\omega : L[g]\vDash S_{\gamma+m} \in \NS \}.
\end{equation}
\end{itemize}
In fact, the definability of $\mathcal G'$ will be \emph{local} to a large class of countable models, namely, the following class.
\begin{definition}
A {\emph{suitable model}} is a transitive  model $\mathcal M$ such that $\mathcal M \vDash\ZF^-$, $(\omega_M)^\M$ exists and
$(\omega_M)^\M=(\omega_M)^{L\cap\M}$ (by $\ZF^-$ we mean an appropriate axiomatization of set theory without the Power Set Axiom).
\end{definition}
We will construct $\mathcal G'$ so that it is definable as follows: For each $g \in S_\infty$, 
$g \in \mathcal G' \iff$ for all countable suitable models $\mathcal M$ containing $g$ as an element, $\mathcal M\vDash``\big(\exists \gamma \in \lim(\omega_M)\big)\;\psi[g]= \{m \in\omega : S_{\gamma+m} \in \NS \}$''.

\medskip

The MCG $\mathcal G$ generated by $\mathcal G'$ will be isomorphic to $\mathbb{F}(A)$ for a set $A\subseteq \omega_M$.
The set $\mathcal G'$ will consist precisely of those elements of $\mathcal G$ whose presentation in $W_A$ has no conjugated proper subwords and no roots (see Notation~\ref{n.words}).
While we can achieve the definability property described above only for $\mathcal G'$, we shall see that this entails that the entire group $\mathcal G$ is $\Pi^1_2$-definable (see Lemma~\ref{l.Pi-1-2} and compare Remarks~\ref{r.subwords} and \ref{r.nr}).

\medskip

To achieve all of the above we proceed in several steps.
Since the proof is long and involved, we present a short roadmap which may also be used as a reference for notation.
\begin{enumerate}
\item\label{roadmap.clubs}\label{roadmap.Y} 
We start with the constructible universe $L$ as the ground model. 
First, we do a preparatory forcing not adding subsets of $\omega$. 
To this end, choose a sequence $\la S_\delta : \delta < \omega_M \ra$ of stationary costationary subsets of $\omega_{M-1}$. We force to add a sequence $\la C_\delta : \delta < \omega_M \ra$ such that $C_\delta$ is a club in $\omega_{M-1}$ which is disjoint from $S_\delta$, ``killing'' the stationarity of $S_\delta$.
Moreover, we force to add a sequence 
\[
\la Y^1_\delta : \delta < \omega_M\ra
\] 
such that  $Y^1_\delta \subseteq \omega_1$ and $Y^1_\delta$ ``locally codes'' $C_\delta$. 
By ``locally coding'' we mean the property \hyperlink{**}{$(**)^1_{\gamma,m}$} below, where $\delta=\gamma+m$ for $m\in\NN$ and $\gamma\in\lim(\omega_M)$. 
To be able to associate each $C_{\gamma+m}$ to the same $\gamma$  for each $m\in\NN$ in a robust, local manner,
we also add a ``code'' $ W^1_\gamma\subseteq\omega_1$ for each $\gamma\in\lim(\omega_M)$.

The forcing that adds $\la C_\delta : \delta < \omega_M\ra$, $\la W^1_\gamma : \gamma \in \limord(\omega_M)\ra$, as well as $\la Y^1_\delta : \delta < \omega_M\ra$ is denoted by $\PP^*_0$, and the $(\PP^*_0, L)$-generic extension is denoted by $V_1$ or by $L[G^*_0]$.

\item\label{roadmap.coding.reals} We force over $V_1$ to add a sequence 
$
\la c^W_\gamma : \gamma \in \limord(\omega_M)\ra
$ 
of reals such that
$c^W_\gamma$ codes $W^1_\gamma$ (and thus, $\gamma$).
We denote the forcing that adds this sequence by 
$\PP^W$ and the $(V_1,\PP^W)$-generic extension by $V_2$.

\item\label{roadmap.continuum} We add $\omega_N$-many Cohen reals using $\operatorname{Add}(\omega,\omega_N)$. Write $V_3$ for the $(V_2,\operatorname{Add}(\omega,\omega_N))$-generic extension. 

\item\label{roadmap.group} 
We now force to add the definable MCG $\mathcal G$. This is done in an iteration $\PP^{\mathcal{G}}:=\la\PP_\alpha,\dot{\QQ}_\alpha:\alpha\in\omega_M\ra$ of length $\omega_M$ over $V_3$, letting $\QQ_\alpha = \QQ^{\mathcal{F}}_\alpha * \QQ^{\text{cd}}_{\mathcal F^*}(\mathcal{F}_\alpha)* \QQ^{\mathcal G}_\alpha * \DD$ be defined in $V_3[G_\alpha]$, the $(V_3, \PP_\alpha)$-generic extension, as follows:
\begin{enumerate}
\item\label{roadmap.group.first} The first forcing $\QQ^{\mathcal{F}}_\alpha$ adds a family $\mathcal F_\alpha$ of size $\omega_1$ consisting of cofinitary permutations of $\omega$.
We do this in such a manner that in the final model $L[G]$ the graphs of any two elements of $\bigcup_{\alpha < \omega_M} \mathcal F_\alpha$ will be almost disjoint.
\item\label{roadmap.coding.reals2} The next forcing  $ \QQ^{\text{cd}}_{\mathcal F^*}(\mathcal{F}_\alpha)$ adds a real
$c^{\mathcal F}_\alpha$ which almost disjointly  codes $\mathcal F_\alpha$ via a definable almost disjoint family $\mathcal F^* \in L$ which remains fixed throughout the iteration.

\item\label{roadmap.group.G} The forcing $\QQ^{\mathcal G}_\alpha$ from the previous section adds a single generator $\sigma_\alpha$ of our MCG $\mathcal G$, using all the machinery added in the previous steps to ensure definability of the resulting group.
\item\label{roadmap.group.last} Finally, $\DD$ is Hechler's forcing to add a dominating real.
\end{enumerate}
The final $(V_3,\PP^{\mathcal{G}})$-generic extension is denoted by $L[G]$.
\end{enumerate}
Step \eqref{roadmap.clubs} is described in Section~\ref{subsection.preparing} below. 
Steps \eqref{roadmap.coding.reals} and \eqref{roadmap.continuum} are described in Section~\ref{subsection.adding.reals}.
Finally steps \eqref{roadmap.group.first}--\eqref{roadmap.group.last} are described in Section~\ref{subsection.group}.

\subsection{Preparing the Universe}\label{subsection.preparing}

Work in the constructible universe $L$.
The preparatory forcing $\PP^*_0$ from Item~\eqref{roadmap.clubs} is an iteration of length $M-1$,  
\begin{equation}\label{e.P^0}
\PP^*_0 = \PP_0^{M-1} * \hdots * \PP^1_0.
\end{equation}
We must describe $\PP_0^{n}$ for $n\in\{M-1,\hdots,1\}$.

\medskip

Fix, for the remainder of this article, a sequence
\begin{equation*}\label{e.bar.D}
\bar{S}=\la S_\delta:\delta < \omega_M\ra
\end{equation*} 
of stationary costationary subsets of $\omega_{M-1}$ consisting of ordinals of cofinality $\omega_{M-2}$ and such that for $\delta\neq\delta'$, $S_\delta \cap S_{\delta'}$ is non-stationary. 
We also ask that $\bar S$ be definable (without parameter) in $L_{\omega_M}$. 

Let $\PP^{\text{cl}}_{\delta}$ denote the forcing which adjoins a closed unbounded subset $C_{\delta}$ of $\omega_{M-1}$ such that $C_{\delta}\cap S_{\delta}=\emptyset$ using bounded approximations.
Note that  $\PP^{\text{cl}}_{\delta}$ preserves stationarity of 
$S_\eta$ for each $\eta \in \omega_M\setminus\{\delta\}$.
We define 
\begin{equation*}\label{e.PP^0_M-1}
\PP_0^{M-1} = \prod_{\delta<\omega_M} \PP^{\text{cl}}_{\delta}
\end{equation*}
where the product uses supports of size at most $\omega_{M-2}$. 
Note that $\PP_0^{M-1}$ adds no sequences of length at most $\omega_{M-2}$ and preserves cardinals and cofinalities.

\medskip

Work in $L$ again momentarily. 
We shall need to code each $C_\delta$, as well as each $\gamma\in\lim(\omega_{M})$ by a subset of $\omega_1$. 
To this end, fix for each $n$ such that $M-1 < n \leq 1$ a sequence $\bar{S}^n=\la S^n_\xi:\xi<\omega_{n+1}\ra$  of subsets of $\omega_{n}$ with pairwise intersections of size at most $\omega_{n-1}$. 
We ask that $\bar{S}^n$ is $\Sigma_1$-definable in $L_{\omega_{M}}$ (without parameters) and that, moreover:

\medskip
\noindent
\hypertarget{p.collapse1}{($\dagger$):} For each $\xi < \omega_{n+1}$, $S^{n+1}_{\xi} \in L_\mu$ where $\mu$ is least such that $L_\mu\vDash |\xi| = \omega_n$.
\smallskip

For the coding of each $\gamma\in\lim(\omega_M)$, we make the following definition.
\begin{definition}~
\begin{itemize}
\item Let $G\colon \ON^2\to\ON$ be G\"odel's $\Sigma_1$ definable pairing function.
Let us call a set \emph{$W\subseteq \omega_{M-1}$ an %
$M-1$-code of $\gamma<\omega_M$} if %
\begin{equation*}\label{e.least-code}
\la \gamma, < \ra \cong \la \omega_{M-1}, G^{-1}[W]\ra.
\end{equation*}
\item For each $\gamma\in\lim(\omega_M)$, let $W^{M-1}_\gamma$ denote its $\leq_L$-least $M-1$-code.
\item Let $\gamma \in\lim(\omega_M)$ and $W\subseteq \omega_n$ for $M-1>n\geq 1$. By induction, define $W$ to be an $n$-code of $\gamma$ if $\{\xi < \omega_{n+1} : \lvert W \cap S^n_\xi \rvert<\omega_n\}$ is an $n+1$-code of $\gamma$.
\end{itemize}
\end{definition}

Until further notice, let us fix $\gamma\in\lim(\omega_{M})$ and $m\in\NN$ and write $\delta=\gamma+m$. 
Work in $L[C_\delta]$. 
We first find $Y^{M-1}_{\delta}\subseteq\omega_{M-1}$ such that
 $Y^{M-1}_\delta$ localizes the non-stationarity of $S_{\gamma+m}$ to suitable models of height at least $\omega_{M-2}$, in the following sense:

\medskip
\noindent
\hypertarget{*}{$(*)_{\gamma,m}$}: Suppose $\M$ is a suitable model with $\omega_{M-2} \subseteq \mathcal M$ and such that for $\beta = (\omega_{M-1})^{\mathcal M}$ we have $\{Y^{M-1}_\delta\cap\beta,W^{M-1}_\delta\cap\beta\}\subseteq\M$.  
Then it must hold that $\mathcal{M}\vDash$``$W^{M-1}_\gamma\cap\beta$ is an $M-1$-code of some $\bar \gamma\in\lim(\omega_M)$ such that $S_{\bar\gamma+m}\in \NS$''. 

\medskip

To see such $Y^{M-1}_\delta$ can be found, momentarily write 
\[
D =\{\beta\in\omega_{M-1} : (\exists \mathcal{M}^*)\; \mathcal M^* \prec L_{\omega_M}[C_\delta], \{C_\delta,\gamma\}\in \mathcal M^*, \beta=\omega_{M-1}\cap \mathcal M^*\}.
\]
For $Y\subseteq\ON$, let $\operatorname{Even}(Y)=\{\xi : 2\xi \in Y\}$ and $\operatorname{Odd}(Y)=\{\xi : 2\xi+1 \in Y\}$. 
Choose $ Y^{M-1}_\delta$ to be any subset of $\omega_{M-1}$ such that $\operatorname{Even}(Y^{M-1}_\delta)=C_\delta$ and for each $\xi \in D$, the preimage under $G$ of $\operatorname{Odd}(Y^{M-1}_\delta)\cap [\xi, \xi+\omega_{M-2})$ is a well-founded binary relation of rank at least $\min(D\setminus (\xi+1))$.

To see that $Y^{M-1}_\delta$ indeed satisfies $(*)_{\gamma,m}$, let $\mathcal M$ as in $(*)_{\gamma,m}$ be given.
 By choice of $Y^{M-1}_\delta$, $\beta=(\omega_{M-1})^{\mathcal M} \in D$. 
 Thus there is a transitive model $\overline{\mathcal M}^*$ and an elementary embedding $j: \overline{\mathcal M}^* \to L_{\omega_M}[C_\delta]$ with critical point $\beta$ and such that $\gamma \in \ran(j)$. Since also $C_\delta$ and $W^{M-1}_\gamma \in \ran(j)$, by elementarity 
 $\overline{\mathcal M}^*\vDash$``$W^{M-1}_\gamma\cap\beta$ is an $M-1$-code of a limit ordinal $\bar\gamma$ such that $C_\delta\cap \beta \cap S_{\bar \gamma +m}=\emptyset$''.
But the rank of $G^{-1}[W^{M-1}\cap\beta]$ is absolute between transitive models, so $\mathcal M$ must satisfy the same sentence.

\medskip

Suppose now that $M-2\geq 1$ and let $n\in\{M-2, \hdots, 1\}$. We will define $\PP_0^n$.
By induction, suppose we are in an $(L,\PP_0^{M-1} * \hdots * \PP_0^{n+1})$-generic extension and we have for each $\gamma \in \lim(\omega_M)$ an $n+1$-code  $W^{n+1}_\gamma$ and for each $m\in\NN$ a set $Y^{n+1}_{\gamma+m}$ which localizes the non-stationarity of $S_{\gamma+m}$ in the following sense: 

\bigskip

\noindent
\hypertarget{**}{$(**)^{n+1}_{\gamma,m}$}: Suppose $\M$ is a suitable model with $\omega_{n} \subseteq \mathcal M$ and for $\beta = (\omega_{n+1})^{\mathcal M}$, we have $\{W^{n+1}_\gamma\cap\beta,Y^{n+1}_\delta\cap\beta\}\subseteq\M$. 
Then
$\mathcal{M}\vDash \varphi(n+1, W^{n+1}_\gamma\cap\beta,m)$, where $\varphi(n+1,W,m)$ is the formula: ``$W$ is a $n+1$-code of some $\bar \gamma\in\lim(\omega_M)$ such that $S_{\bar\gamma+m}$ is not stationary''. 

\bigskip

For the induction start, note that for $n+1=M-1$ this is indeed satisfied since \hyperlink{**}{$(**)^{M-1}_{\gamma,m}$} is just \hyperlink{*}{$(*)_{\gamma,m}$}. 

Write $\PP^{\text{cd}}_{{n}}(X)$ to denote Solovay's almost disjoint coding to add a subset of $\omega_{{n}}$ which codes $X$ via $\bar S^{{n}}$.
Forcing with $\PP^{\text{cd}}_{{n}}(W^{n+1}_\gamma)$ adds generic ${n}$-code for $\gamma$ which we denote by $W^{{n}}_\gamma$.
For each $m\in\NN$, forcing with $\PP^{\text{cd}}_{{n}}(Y^{n+1}_{\gamma+m})$ adds a set a $Z^{{n}}_{\gamma+m}\subseteq \omega_{{n}}$ which via $\bar S^{{n}}$ codes $Y^{n+1}_{\gamma+m}$. 

Crucially, if $\M$ is a suitable model such that $\{W^{{n}}_\gamma,Z^{{n}}_{\gamma+m}\}\subseteq\M$ (and hence $\omega_{{n}} \subseteq \mathcal M$)
then $\{W^{n+1}_\gamma\cap\beta,Y^{n+1}_\delta\cap\beta\}\subseteq\M$ for $\beta=(\omega_{n+1})^{\mathcal M}$ as a consequence of  
\hyperlink{p.collapse1}{($\dagger$)}
and hence by 
\hyperlink{**}{$(**)^{n+1}_{\gamma,m}$},
 $\mathcal{M}\vDash$``$W^{{n}}_\gamma$ is a ${n}$-code of a limit ordinal $\bar \gamma<\omega_M$ such that $S_{\bar\gamma+m}$ is not stationary''. In other words, $\mathcal M \vDash\varphi({n},W^{{n}}_\gamma,m)$.

\medskip

The next step is to %
localize the coding to suitable models of height at least $\omega_{n-1}$.

\begin{definition} 
Fix $m\in\NN$ and $W, Z\subseteq\omega_{{n}}$ such that $\varphi({n},W,m)$  holds in
any suitable model $\mathcal{M}$ containing both $W$ and $Z$ as elements. Denote by $\mathcal{L}_{{n}}(Z,W,m)$ the forcing whose conditions are 
functions $r:|r|\to 2$ with domain $|r| \in \omega_{{n}}$ such that
\begin{enumerate}
\item if $\xi<|r|$ then $\xi\in Z$ iff $r(2\cdot \xi)=1$,
\item if $\beta\leq|r|$ and $\mathcal{M}$ is a suitable model containing $r\rest \beta$ as an element and with
 $\beta=\omega_{n}\cap \mathcal{M}$, then 
$\mathcal{M}\vDash\varphi({n},W\cap \beta,m)$.
\end{enumerate}
The order on $\mathcal{L}_{{n}}(Z,W,m)$ is end-extension.
\end{definition}

As is easy to verify, the forcing $\mathcal{L}_{{n}}(Z^{{n}}_{\delta},W^{{n}}_\gamma,m)$ adds the characteristic function of a set $Y^{{n}}_{\delta}\subseteq\omega_{{n}}$ such that $\hyperlink{**}{(**)^{{n}}_{\gamma,m}}$ is true.

\medskip

Define the next iterand in \eqref{e.P^0} to be
\[
\PP_0^{n} = \prod_{\gamma\in\lim(\omega_{M})}\bigg( \PP^{\text{cd}}_{{n}}(W^{n+1}_\gamma) * 
 \prod_{m\in\NN}  \PP^{\text{cd}}_{{n}}(Y^{n+1}_{\gamma+m}) * \mathcal{L}_{{n}}(Z^{{n}}_{\gamma+m},W^{{n}}_\gamma,m)\bigg)
\]
with supports of size at most $\omega_{n-1}$.
It is well-known that $\PP^n_0$ adds no sequences of length at most $\omega_{n-1}$, and preserves cofinalities and cardinals when forcing over the $(L,\PP_0^{M-1} * \hdots * \PP_0^{n})$-generic extension (see \cite{memoirs,VFSFLZ11}).

\medskip

Denote by $V_0$ the $(L,\PP^*_0)$-generic extension. 
In $V_0$ we have $W^1_\gamma$ coding $\gamma$ via $\bar S^1, \hdots, \bar S^{M-1}$ for each $\gamma \in \lim(\omega_M)$, and $\la Y^1_\delta : \delta<\omega_M\ra$ such that
for each $\gamma \in \lim(\omega_M)$ and $m\in\NN$, \hyperlink{**}{$(**)^1_{\gamma,m}$} holds true.
Note that $V_0\cap\omega^\omega=L\cap\omega^\omega$.

\subsection{Preparing the continuum}\label{subsection.adding.reals}

Fix (for the rest of the proof) a constructible almost disjoint family
$$\mathcal{F}^*:=\{a_{\xi}: \xi < \omega_1\}$$
which is $\Sigma_1$ (without parameters) in $L_{\omega_2}$ and with the following property (which will be important in the proof of Lemma~\ref{l.MCG.coding.local}):

\medskip
\noindent
\hypertarget{p.collapse}{($\ddagger$):} For each $\xi < \omega_1$, $a_{\xi} \in L_\mu$ where $\mu$ is least such that $L_\mu\vDash |\xi| = \omega$.

\medskip

Next force with the finite support iteration 
\[
\PP^W:=\la\PP^{W}_\gamma,\dot{\QQ}^{W}_\gamma:\gamma\in\limord(\omega_M)\}
\]
where
for each $\gamma$, $\dot \QQ^W_\gamma$ adds the real $c_\gamma^W$ which almost disjointly via the family $\mathcal F^*$ codes 
$W^1_\gamma$. 
Let $V_2$ be the $(\dot \PP^W,V_0)$-generic extension .

Using the standard forcing $\operatorname{Add}(\omega,\omega_N)$ (finite partial functions from $\omega_N \times \omega$ into $2$) adjoin $\omega_N$-many reals to $V_2$ to increase the size of the continuum to $\omega_N$ and denote the resulting model  to obtain a model $V_3$.

\subsection{Adding the MCG}\label{subsection.group}

\medskip

We shall now define a finitely supported iteration $\PP^{\mathcal{G}}:=\la\PP_\alpha,\dot{\QQ}_\alpha:\alpha\in\omega_M\ra$ which adds a self-coding MCG to the model $V_3$.

Along the iteration, for each $\alpha\in\omega_M$ we will define a $\PP_\alpha$-name $\dot{I}_\alpha\subseteq [\beta_\alpha,\beta_{\alpha+1})$ for a set of ordinals, such that at stage $\alpha$ of the construction we adjoin reals encoding a stationary kill of $S_\delta$ (that is, a real locally coding $C_\delta$) for $\delta\in I_\alpha$.
We then show that there is ``no accidental coding of a stationary kill'' in Lemma~\ref{no_accidental_real}.

In order to define $\PP^{\mathcal{G}}:=\la\PP_\alpha,\dot{\QQ}_\alpha:\alpha\in\omega_M\ra$, recall from \eqref{e.comp.bij} that we have fixed a computable bijection 
\[
\psi:\omega\times\omega\to\omega
\]
and in addition, fix a bijection
\[
\psi':\omega_1\times\omega\times\omega\to\omega_1
\] 
 which is $\Delta_1$ (without parameters) provably in $\ZF^-$.
The function $\psi'$ will be used to identify the family $\mathcal{F}_\alpha$ which we add at stage $\alpha$ with a subset of $\omega_1$. 

Suppose now by induction we are in  $V_3[G_\alpha]$, the $(V_3, \PP_\alpha)$-generic extension.
We presently define $\QQ_\alpha = \QQ^{\mathcal{F}}_\alpha * \QQ^{\text{cd}}_{\mathcal F^*}(\mathcal{F}_\alpha)* \QQ^{\mathcal G}_\alpha * \DD$.

For the definition of $\QQ^{\mathcal{F}}_\alpha$ assume by induction that at previous stages we have added families $\mathcal F_\beta$ for $\beta < \alpha$ 
consisting of cofinitary permutations. 
We now adjoin a family 
 \[
 \mathcal{F}_\alpha=\la f^\alpha_{m,\xi} : m\in\NN,\xi\in\omega_1\ra
 \] 
 of permutations 
 such that $\lvert f^{\alpha}_{m,\xi} \cap f^{\alpha'}_{m',\xi'}\rvert < \omega$ for each $(m,\xi) \in \NN\times\omega_1$ and $(\alpha',m',\xi') \in (\alpha+1) \times\NN\times\omega_1$  with $(\alpha,m,\xi)\neq(\alpha',m',\xi')$. 
 For this we can use a $\sigma$-centered forcing, namely a variant of Solovay's forcing to add almost disjoint sets with finite conditions which is defined in~\cite{VFAT}. Denote this forcing by $\QQ^{\mathcal{F}}_\alpha$ and by $V_{\alpha,1}$ the resulting model. 
 
 Next let $ \QQ^{\text{cd}}_{\mathcal F^*}(\mathcal{F}_\alpha)$ be the forcing to add a real $c^{\mathcal F}_\alpha$ which almost disjointly via the family $\mathcal F^*$ (see Section~\ref{subsection.adding.reals}) codes 
\[
\psi'\left[\bigcup_{\xi<\omega_1} \{\omega \cdot \xi + m\}\times f^\alpha_{m,\xi}\right],
\]
a subset of $\omega_1$ which via $\psi'$ codes $\mathcal F_\alpha$.
Let $V_{\alpha,2}$ be the extension of $V_{\alpha,1}$ which contains $c^{\mathcal F}_\alpha$
and note that $\mathcal F^\alpha$ is definable in $L[c^{\mathcal F}_\alpha]$.

\medskip

Working in $V_{\alpha,2}$ we define $\QQ^{\mathcal G}_\alpha$, the forcing which adds a new group generator.

Suppose by induction that $\PP_\alpha$ has added a cofinitary representation $\rho_\alpha$.
Its image generates a cofinitary group $\mathcal{G}_\alpha$.  
Suppose by induction that $\dom(\rho_\alpha)=\{\beta_\xi : \xi<\alpha\}\subseteq\lim(\omega_M)$ and write $A_\alpha$ for $\{\beta_\xi : \xi<\alpha\}$.
Moreover write $\rho_\alpha(\beta_\xi)=\sigma_\xi$ for each $\xi<\alpha$.
Our next forcing will add the generic permutation $\sigma_\alpha$ as a new generator,
and we extend our group to $\mathcal G_{\alpha+1}=\la\mathcal G_\alpha\cup\{\sigma_\alpha\}\ra$.

If $\alpha$ is a limit, let 
\[
\beta_\alpha = \sup\{\beta_\xi : \xi<\alpha\}
\]
and otherwise,
let 
\[
\beta_\alpha=\beta_{\alpha-1} + \max(|\alpha-1|, \omega)
\]
(we mean ordinal addition). 
This is the ordinal to which we associate the new generator $\sigma_\alpha$ so that 
\[
\rho_{\alpha+1}=\rho_{\alpha}\cup\{(\beta_\alpha,\sigma_\alpha)\}
\] 
is a cofinitary representation.
 
Writing $a=\beta_\alpha$, every element of $\mathbb{F}(A_\alpha\cup\{a\})$ corresponds to a reduced word 
in $W_{A_\alpha,a}$.
Let $\Wa_\alpha$ be the set of such words in which $a$ or $a^{-1}$ {\emph{occurs at least once}}.
Note that the set $\Wa_\alpha$ corresponds to the new permutations in the group $\mathcal{G}_{\alpha+1}$. 
In other words, permutations in $\mathcal{G}_{\alpha+1}\setminus\mathcal G_\alpha$ are precisely those  of the form $\rho_{\alpha+1}(w)$ with $w\in \Wa_\alpha$, 
and $\rho_{\alpha+1}(w)$ equals $w[\sigma_\alpha]$ in our earlier notation, provided we take $\rho=\rho_\alpha$. 

As before write $\Wncs_\alpha$ for the set of words in $\Wa_\alpha$ without a proper conjugated subword.
Write $\Wnr_\alpha$ for the set of words in $\Wa_\alpha$ which are not of the form $\bar w^n$ for $\bar w \in \Wa_\alpha\setminus\{w\}$. (Compare Notation~\ref{n.words}.)

Let $i_\alpha: \Wncs_\alpha \cap\Wnr_\alpha \to \limord(\lvert \alpha \rvert)$ be a bijection sending $a$ to $0$; 
we shall use $i_\alpha$ to associate the ordinal $\beta_\alpha + i_\alpha(w)$ to each $w \in \Wncs_\alpha$.
We note that those elements of $\mathcal G_{\alpha+1}\setminus \mathcal G_\alpha$ which correspond via $\rho_{\alpha+1}^{-1}$ to words in
$\Wncs_\alpha$ will be associated to ordinals in $[\beta_\alpha,\beta_{\alpha+1})$, and in fact $\sigma_\alpha$ is associated to $\beta_\alpha$. 
Elements of $\mathcal G_{\alpha+1}$ which are not of the form $\rho_{\alpha+1}(w)$ 
 for $w\in\Wncs_\alpha \cap \Wnr$ can be ignored for the purpose of coding (see Lemma~\ref{l.Pi-1-2} and Remarks~\ref{r.subwords} and \ref{r.nr}).

\medskip

For each $w\in \Wncs_\alpha\cap\Wnr_\alpha$ it is the pattern of stationarity on the block  of $\bar S$ consisting of the next $\omega$ ordinals after $\beta_\alpha + i_{\alpha}(w)$ that will code $w$ (i.e., 
we set $\gamma$ in \eqref{e.def.G'} equal to $\beta_\alpha + i_{\alpha}(w)$). Let, for $w\in \Wncs_\alpha\cap\Wnr_\alpha$, 
\begin{equation}\label{e.z^w}
z^w = \{ 2^m : m \in c^{\mathcal F}_{\alpha}\}\cup\{3^m : m \in c^W_{\beta_\alpha + i_{\alpha}(w)}\}
\end{equation}
(recalling that the reals $c^W_\gamma$ were constructed in Section~\ref{subsection.adding.reals}) and
define  
\[
\bar z = \la z^w : w\in \Wncs_\alpha \cap \Wnr_\alpha\ra.
\]
Further, define 
\[
Y^w_m = Y^1_{\beta_\alpha + i_\alpha(w) + m}
\] 
 for each $w \in \Wncs_\alpha\cap\Wnr_\alpha$ (the sets $Y^1_\delta$ where constructed in Section~\ref{subsection.preparing}) and let
\[
\mathcal Y = \la Y^w_m : w \in \Wncs_\alpha \cap \Wnr_\alpha, m\in\omega  \ra.
\] 
With the notation from Section~\ref{section.group.forcing} we now define 
\[
\QQ^{\mathcal G}_\alpha = \QQ^{\mathcal{F}_\alpha,\mathcal{Y},\bar z}_{\rho_\alpha,\{\beta_\alpha\}}.
\]
In Proposition~\ref{prop.group.forcing} we have seen that $\QQ^{\mathcal G}_\alpha$ is Knaster and adjoins a new generator $\sigma_{\alpha}$ such that the following properties hold:

\begin{enumerate}[label=(\Roman*$_\alpha$)]
\item\label{i.A.alpha}  The group $\la \im(\rho_\alpha)\cup\{\sigma_\alpha\}\ra$ is cofinitary.

\item\label{i.B.alpha} If $f\in V^{\PP_\alpha}\backslash\mathcal{G}_\alpha$ is a permutation which is not covered by finitely many members
of $\mathcal{F}_\alpha$ and  $\la\mathcal{G}_\alpha\cup\{f\}\ra$ is cofinitary,  
then for infinitely many $k$, $f(k)=\sigma_\alpha(k)$. 
This property will eventually ensure maximality of $\bigcup_{\alpha<\omega_M}\mathcal{G}_\alpha$.

\item\label{i.C.alpha} For each $w\in \Wncs_\alpha \cap\Wnr_\alpha$ there is $m_w\in \omega$ such that
for all $k\in\omega$, $w^{3k}[\sigma_\alpha](m_w)\in P_{S(\chi\res k)}$, where $\chi$ is the characteristic function of $z$. That is, $w[\sigma_{\alpha}]$
encodes $\mathcal{F}_\alpha$ via the real $c^{\mathcal F}_\alpha$ as well as $W^1_{\beta_\alpha+i_\alpha(w)}$ via the real $c^W_{\beta_\alpha+i_\alpha(w)}$.
\item\label{i.D.alpha} For each $w\in \Wa_\alpha$, for all $m\in\psi\big[w[\sigma_\alpha]\big]$, for all $\xi\in\omega_1$
$$|w[\sigma_\alpha]\cap f^{\alpha}_{m,\xi}|<\omega\text{ iff }\xi\in Y^w_{m}.$$
That is, $w[\sigma_\alpha]$ codes $Y^w_m$ for each $m\in\psi\big[w[\sigma_\alpha]\big]$ via almost disjoint coding.
\item\label{i.E.alpha} The cofinitary representation $\rho_{\alpha+1}=\rho_\alpha\cup\{(\beta_\alpha,\sigma_\alpha)\}$ 
is sufficiently generic.
\end{enumerate}

As we are going to see in the next section, property \ref{i.D.alpha}  implies that the new permutation $w[\sigma_\alpha]$ encodes itself
via stationary kills on the segment $\la S_\delta:\beta_\alpha+i_\alpha(w)\leq \delta<\beta_\alpha+i_\alpha(w)+\omega\ra$. Furthermore, these stationary kills are witnessed by countable suitable models containing $w[\sigma_\alpha]$.

Let  $\dot I_\alpha$ be a $\bar\PP^{\mathcal G}_{\alpha+1}$-name for
\[
I_\alpha=\big\{\beta_\alpha+i_\alpha(w)+m: w\in\Wncs_\alpha \cap\Wnr_\alpha, m\in\psi\big[w[\sigma_\alpha]\big]\big\}. 
\]
Thus $I_\alpha$ denotes the set of indices of the
stationary sets for which we explicitly adjoin reals encoding a stationary kill at stage $\alpha$ of the iteration. 
Note that $\beta_\alpha=\sup I_\alpha$. 

\medskip

The final forcing of the $\alpha$th step of the iteration, $\mathbb D$, is simply Hechler's $\sigma$-centered forcing to add a dominating real. 
With this the definition of $\QQ_\alpha = \QQ^{\mathcal{F}}_\alpha * \QQ^{\text{cd}}_{\mathcal F^*}(\mathcal{F}_\alpha)* \QQ^{\mathcal G}_\alpha * \DD$, and thus the description of the forcing iteration to add a co-analytic MCG of intermediate size, is complete.

\section{Definability and maximality of the group}\label{s.definability.maximality}

Forcing with $\PP^{\mathcal{G}}$ over $V_3$ we obtain a generic
  $G$ over $L$ for the forcing 
\[
\PP:=\PP^*_0  * \PP^W * \operatorname{Add}(\omega,\omega_N) *\PP^{\mathcal{G}}
\]
recalling that $\PP^*_0$ added the sets $W^1_\gamma$ and $Y^1_{\gamma+m}$, and $\PP^W$ added  reals $c^W_\gamma$ ``locally coding'' the ordinal $\gamma$ for each $\gamma \in \limord(\omega_M)$; $\operatorname{Add}(\omega,\omega_N)$ made $2^\omega = \omega_N$; and finally $\PP^{\mathcal{G}}$ added a generic self-coding subgroup of $S_\infty$ and a sequence of dominating reals.
Also recall that all the forcings after $\PP^*_0$ are Knaster, and $\PP^*_0$ did not add any countable sequences.  

Work in $L[G]$ from now on. 
First we must show that no real codes an ``accidental'' stationary kill.

\begin{lemma}\label{no_accidental_real} For each $\bar\delta$ which is not in
$I= \bigcup\{I_\alpha:\alpha<\omega_M\}$
there is no real in $L[G]$ coding a stationary kill of $S_{\bar\delta}$, i.e., there is no $r\in \powerset(\omega) \cap L[G]$ such that $L[r]\vDash S_{\bar\delta} \in \NS$.
\end{lemma}
\begin{proof}
The argument has precursors in~\cite[Lemma~8.22]{memoirs} and~\cite[Lemma~3]{VFSFLZ11}.
Let $\dot{I}$ be a name for $I$ and suppose that $p \forces \check {\bar\delta} \notin \dot I$. 
Write ${\bar\delta} = \bar\gamma + \bar m$ with $\bar \gamma$ a limit ordinal and $\bar m\in\NN$.
Working in $L$, let
\begin{equation*}
\PP^{\neq{\bar\delta}}_0 =    \PP_{\neq{\bar\delta}}^{M-1} * \hdots * \PP_{\neq{\bar\delta}}^1.
\end{equation*}
where 
\[
\PP_{\neq{\bar\delta}}^{M-1}=\prod_{\delta\in\omega_M\setminus\{\bar\delta\}} \PP^{\text{cl}}_{\delta} 
\]
with supports of size at most $\omega_{M-2}$, 
and by finite reverse induction on $n \in \{M-2, \hdots, 1\}$
we define, in the $(L,\PP_{\neq\bar\delta}^{M-1} * \hdots \PP_{\neq\bar\delta}^{n-1})$-generic extension,
\begin{multline*}
\PP_{\neq\bar\delta}^{n}= 
\bigg( \PP^{\text{cd}}_{n}(W^{n+1}_{\bar\gamma}) * 
 \prod_{m\in\NN\setminus\{\bar m\}}  \PP^{\text{cd}}_{n}(Y^{n+1}_{\bar\gamma+m}) * \mathcal{L}_{n}(Z^{n}_{\bar\gamma+m},W^{n}_{\bar\gamma},m) \bigg) \times
 \\
 \prod_{\gamma\in\lim(\omega_{M})\setminus\{\bar \gamma\}}\bigg( \PP^{\text{cd}}_{n}(W^{n+1}_\gamma) * 
 \prod_{m\in\NN}  \PP^{\text{cd}}_{n}(Y^{n+1}_{\gamma+m}) * \mathcal{L}_{n}(Z^{n}_{\gamma+m},W^{n}_\gamma,m)\bigg)
\end{multline*}
with supports of size at most $\omega_{n-1}$.

Moreover, working in the $(L,\PP^{\neq\bar\delta}_0)$-generic extension, let
\begin{equation*}
\PP^{\bar\delta}_0 =    \PP_{\bar\delta}^{M-1} * \hdots * \PP^1_{\bar\delta}
\end{equation*}
where $\PP_{\bar\delta}^{M-1}=(\PP^{\textup{cl}}_{\bar\delta})^L$,  and by induction on $n \in \{M-2, \hdots, 1\}$,
\begin{equation*}
\PP_{\bar\delta}^{n} =     \PP^{\text{cd}}_{n}(Y^{n+1}_{\bar\delta})^{L[W^{n}_{\bar\gamma},Y^{n+1}_{\bar\delta}]} * \mathcal{L}_{n}(Z^{n}_{\bar\delta},W^{n}_{\bar\delta},m)^{L[W^{n}_{\bar\gamma},Z^{n}_{\bar\delta}]}.
\end{equation*}
Since $\PP_{\neq{\bar\delta}}^n$ does not add any subsets to $\omega_{n}$, it is routine to verify that  
\[
\PP^*_0 = \PP^{\neq\bar\delta}_0 * \PP^{{\bar\delta}}_0
\] 
up to equivalence of forcing notions.
Decompose the $\PP^*_0$-generic $G^*_0$ as $G^*_0 = G^{\neq \bar\delta}_0 * G^{{\bar\delta}}_0$ 
where $G^{\neq \bar\delta}_0$ is $(L,\PP^{\neq \bar\delta}_0)$-generic and $G^{{\bar\delta}}_0$ is $(L[G^{\neq \bar{\delta}}_0],\PP^{{\bar\delta}}_0)$-generic.
Working in $L[G^*_0]=L[G^{\neq {\bar\delta}}_0][G^{\bar\delta}_0]$ and writing
\[
\PP' = \big(\PP^W*\operatorname{Add}(\omega,\omega_N) *\PP^{\mathcal{G}} \big)\upharpoonright p
\] 
for the quotient $\PP / \PP^*_0$ below $p$,
it is easy to verify that $\PP' \in L[G^{\neq {\bar\delta}}_0]$ since the iteration never uses $Y^1_{\bar\delta}$.
Thus letting $G'$ be shorthand for the $\PP'$-generic, we may decompose $G$ as $G = G^{\neq {\bar\delta}} * (G'  \times G^{\bar\delta}_0)$.

Let $r$ be any real in $L[G] = L[G^{\neq {\bar\delta}}_0][G'][G^{\bar\delta}_0]$ and write
\[
V_* = L[G^{\neq {\bar\delta}}_0]
\]
We show that 
$
r \in V_*[G']= L[G^{\neq {\bar\delta}}_0][G']
$. 
For this, note that $\PP^{\bar\delta}_0$ adds no countable sequences  over $V_*$ (as even $\PP^*_0$ adds no such sequences to $L$). Moreover, by  the argument from Lemma~\ref{l.Knaster}, $\PP'$ remains ccc in $V_*[G^{\bar\delta}_0]$. It follows (see \cite[Claim, p.~9]{VFSFLZ11}) that $\PP^{\bar\delta}_0$ also adds no countable sequences over $V_*[G']$. So indeed $r \in V_*[G']$.
But since $\PP^{\neq {\bar\delta}} * \PP'$ preserves stationarity of $S_{\bar\delta}$, the latter is still stationary in $ V_*[G']=L[G^{\neq {\bar\delta}}_0][G']$ and hence in $L[r]$.
\end{proof}

Let
$\mathcal{G}$ be the group generated by $\{\sigma_\alpha:\alpha\in\omega_M\}$. 
Given $w \in \Wa_\alpha$, we write $w^G$ for $\rho_\alpha(w)$, i.e., the interpretation of $w$ that replaces every generator index $\beta_\gamma$ for $\gamma\leq\alpha$ by the corresponding generic permutation $\sigma_\gamma$. 
We shall occasionally use the notation $w[\sigma_\alpha]$ for $w^G$; this is consistent with our previous use of this notation.

\begin{lemma} %
The group $\mathcal{G}$ is a maximal cofinitary group.
\end{lemma}
\begin{proof}
By property \ref{i.A.alpha} of the iterands $\dot{\QQ}_\alpha$ 
the group $\mathcal{G}$ is cofinitary. It remains to show maximality. Suppose by contradiction that $\mathcal{G}$ is not maximal. Then there is a cofinitary permutation $h\notin \mathcal{G}$ such that the group generated by $\mathcal{G}\cup \{h\}$ is cofinitary. 
Write $G_\alpha$ for the $(L,\PP^*_0  * \PP^W * \operatorname{Add}(\omega,\omega_N) *\PP_\alpha)$-generic induced by $G$.
Find $\beta$ such that $h\in L[G_\beta]$.
 Then there is $\beta'\in\{\beta, \beta+1\}$ such that $h$ is not a subset of the union of finitely many members of $\mathcal{F}_{\beta'}$:
For otherwise by the pigeonhole principle we find $f \in \mathcal F_\beta$ and $f' \in \mathcal F_{\beta+1}$ such that $\lvert f \cap f' \rvert = \omega$, contradicting the choice of $\mathcal F_\beta$ and $\mathcal F_{\beta+1}$.
Letting $\alpha = \beta'+1$, by property \ref{i.B.alpha}
of the forcing $\QQ^{\mathcal G}_\alpha$
in $L[G_\alpha]$, the generic permutation $\sigma_\alpha$ infinitely  often takes the same value as $h$, and so $\sigma_\alpha\circ h^{-1}$ is not cofinitary, which is a contradiction.
\end{proof}

It remains to show that $\mathcal G$ is $\Pi^1_2$.

\begin{lemma}\label{lemma1}
Let %
$g \in S^\infty \cap L[G]$.
Then $g = w^G$ for some $w \in \bigcup_{\alpha<\omega_M}\Wncs_\alpha \cap\Wnr_\alpha$ if and only if there is $\gamma\in\limord(\omega_M)$ such that
\begin{equation*}\label{e.lemma1}
\psi[g] = \{m \in \omega :  L[g]\vDash S_{\gamma+m} \in \NS\},
\end{equation*}
or equivalently, such that 
$\psi[g] =\{m \in \omega :  (\exists r \in\powerset(\omega)) \; L[r]\vDash S_{\gamma+m} \in \NS\}$.
\end{lemma}
\begin{proof}
Suppose $g = w^G$ and $\alpha$ is minimal such that $w \in \Wncs_\alpha \cap\Wnr_\alpha$.
We prove the lemma for $\gamma=\beta_\alpha+i_\alpha(w)$. 
By property \ref{i.C.alpha} of the forcing $\QQ^{\mathcal G}_\alpha$ the permutation $g$ codes $z^w$ and therefore 
\[
\mathcal F_{\beta_\alpha+i_\alpha(w)} \in L[g].
\]
By property \ref{i.D.alpha} of the forcing $\QQ^{\mathcal G}_\alpha$ the real $g$ almost disjointly codes via the family $\mathcal{F}_\alpha$ the set $Y^1_{\beta_\alpha+i_\alpha(w)+m}$ for each $m\in \psi[g]$. However $Y^1_{\beta_\alpha+i_\alpha(w)+m}$ codes  $C_{\beta_\alpha+i_\alpha(w)+m}$, the club set disjoint from $S_{\beta_\alpha+i_\alpha(w)+m}$. 

If $m\notin\psi[g]$, then
$\beta_\alpha+i_\alpha(w)+m\notin I_\alpha$ and so by Lemma~\ref{no_accidental_real}, there is no real $r$ in $L[G]$ such that $L[r]\vDash S_{\beta_\alpha+i_\alpha(w)+m}\in\NS$.

For the other direction, suppose $\gamma\in\limord(\omega_M)$ is such that  for all $m\in\omega$, 
$\;L[r]\vDash S_{\gamma+m} \in \NS$ for some $r \in\powerset(\omega)$ if and only if $m\in\psi[g]$.
Then by Lemma~\ref{no_accidental_real}, $\psi[g] \subseteq \{m \in \omega : \gamma + m \in  I_\alpha\}$. 
By the previous, the latter set equals $\psi[w^G]$ where $w$ is such that $\beta_\alpha + i_{\alpha}(w)=\gamma$ for some $\alpha<\omega_M$.
So $g=w[\sigma_\alpha]=w^G$.
\end{proof}

\begin{lemma}\label{l.MCG.coding.local}
Let $g=w^G$ for some %
$w \in \Wncs_\alpha \cap\Wnr_\alpha$ with $\alpha<\omega_M$.
Then for every countable suitable model ${\mathcal{M}}$ such that $g\in\mathcal{M}$ there is a limit ordinal $\bar\gamma<(\omega_M)^{\mathcal{M}}$ such that
$$\mathcal{M}\vDash \psi[g]= \{m \in\omega : S_{\bar\gamma+m} \in \NS \}.$$
\end{lemma}
\begin{proof}
Suppose ${\mathcal{M}}$ is a countable suitable model such that $g\in\mathcal{M}$. 
Let $\alpha$ be minimal so that $w\in \Wncs_\alpha \cap\Wnr_\alpha$ and let $\gamma = i_{\alpha}(w)$.
Since $w^G=w[\sigma_\alpha]$ encodes $z^w$ (by property \ref{i.C.alpha} of $\QQ^{\mathcal G}_\alpha$ and minimality of $\alpha$) and by \hyperlink{p.collapse}{($\ddagger$)}, we have that 
\[
\{f^\alpha_{m,\xi} : m\in\omega, \xi < \beta\} \in\mathcal{M}
\] 
where $\beta=(\omega_1)^{\mathcal{M}}$.
By property \ref{i.D.alpha}, $g=w^G$ almost disjointly codes $Y^1_{\gamma+m}\cap\beta$ for each $m\in\psi[g]$ and hence $Y^1_{\gamma+m}\cap \beta\in\mathcal{M}$. 
Similarly, $W^1_{\gamma} \cap \beta\in \mathcal{M}$. 
Then for each $m\in\psi[g]$, by $\hyperlink{**}{(**)^1_{\gamma,m}}$ the sentence $\varphi(1,W^1_{\gamma}\cap\beta,m)$ holds in the suitable model $L[W^1_{\gamma} \cap \beta,Y^1_{\gamma+m}\cap \beta]^{\mathcal{M}}$ where $\beta = (\omega_1)^{\mathcal{M}}$.
Hence  in $\mathcal M$, the following holds:

The set $W^1_{\gamma}\cap\beta$ almost disjointly codes via the sequences $\bar S^1$, \dots, $\bar S^{M-1}$ a code for an ordinal $\bar \gamma < (\omega_M)^{\mathcal M}$ such that $S_{\bar\gamma+m}$ is not stationary.
As $m\in \psi^{-1}[g]$ was arbitrary and $\bar \gamma$ depends only on $\mathcal M$ and $\gamma$---which in turn depends only on $w$---and not on $m$,
 indeed $\bar\gamma$ witnesses that the lemma holds.
\end{proof}

\begin{lemma}\label{l.equidef} Let %
$g$ be a real such that for every countable suitable model $\mathcal{M}$ containing $g$ as an element there is $\bar\gamma<(\omega_{M})^{\mathcal{M}}$ such that
$$\mathcal{M}\vDash \psi[g]= \{m \in\omega : S_{\bar\gamma+m} \in \NS\}.$$
Then for some $\alpha < \omega_M$, $g = w^G = w[\sigma_\alpha]$ where $w \in \Wncs_\alpha\cap\Wnr_\alpha$. 
\end{lemma}
\begin{proof}
By L\"owenheim-Skolem  take a countable elementary submodel $\mathcal{M}_0$ of $L_{\omega_{M+1}}[g]$ such that $g \in \mathcal{M}_0$ and let $\mathcal{M}$ be the unique transitive model isomorphic to $\mathcal{M}_0$.
Then by assumption 
$$\mathcal{M}\vDash \big(\exists \bar\gamma\in\limord(\omega_{M})\big)\;\psi[g]= \{m \in\omega :S_{\bar\gamma+m}\;\text{is non-stationary}\}$$
so by elementarity the same holds with $\mathcal{M}$ replaced by $L_{\omega_{M+1}}[g]$, and hence for some $\gamma \in\limord(\omega_M)$
$$L[g]\vDash \psi[g]= \{m \in\omega : S_{\gamma+m} \in \NS \}.$$
By Lemma~\ref{lemma1} it follows that $g$ is of the desired form. 
\end{proof}

\begin{lemma}\label{l.Pi-1-2}
The MCG $\mathcal G$ is $\Pi^1_2$ in $L[G]$.
\end{lemma}
\begin{proof}
Recall that we denote by $\sigma_0$ the first generator added by $\QQ^{\mathcal G}_0$ over $V_3$. 
Note first that $g \in \mathcal G$ if and only if there is $k\in \omega$, $\alpha < \omega_M$, and $w \in \Wncs_\alpha \cap\Wnr_\alpha$   such that $(\sigma_0)^k g = w^G$. 

By the previous lemmas, $g \in \mathcal G$ if and only if $g \in S_\infty$ and the following statement $\Phi(g)$ holds: 
For every suitable countable model $\mathcal{M}$
\emph{if} $g \in \mathcal M$ and for some $\sigma_* \in \mathcal M \cap S_\infty$
$$\mathcal{M}\vDash \psi[\sigma_*]=\{m\in\omega: S_{0+m}\text{ is non-stationary}\},$$
that is, if $g^*$ is the first generator $\sigma_0$, \emph{then} for some $k\in\omega$
$$\mathcal{M}\vDash \big(\exists \gamma \in \limord(\omega_M)\big)\;\psi\big[(\sigma_*)^k g\big]=\big\{m\in\omega : S_{\gamma+m}\in \NS\big\}.$$
It is standard to see $\Phi(g)$ can be expressed by a $\Pi^1_2$ formula.
\end{proof}

Thus we obtain our main result:

\begin{theorem} Let $2\leq M < N < \aleph_0$ be given. There is a cardinal preserving generic extension of the constructible universe $L$ in which $$\mathfrak{a}_g=\mathfrak{b}=\mathfrak{d}=\aleph_M<\mathfrak{c}=\aleph_N$$
	and in which there is a $\Pi^1_2$ definable maximal cofinitary group of size $\mathfrak a_g$.
\end{theorem}
\begin{proof} It remains to see that in the final model $L[G]$ there are no maximal cofinitary groups of cardinality strictly smaller than $\aleph_M$. 
Recall from the beginning of Section~\ref{s.iteration}, Item (\ref{roadmap.group.last}) that each for $\alpha<\omega_M$, $\QQ_\alpha$ adds a dominating real over $V_3[G_\alpha]$. Thus in the final model $L[G]$ there is a scale of length $\omega_M$ and so $\mathfrak{b}=\mathfrak{d}=\aleph_M$. Since 
$\mathfrak{b}\leq\mathfrak{a}_g$ we obtain $\mathfrak{a}_g=\aleph_M$.
\end{proof}

\section{Questions}\label{s.questions}
In this section, we state some of the remaining open questions.
\begin{enumerate}
\item Can one construct in $\ZFC$ (or just $\ZF$) a countable cofinitary group which can not be enlarged to a Borel MCG? 
Recall here that in contrast to, e.g., MAD families, there \emph{are} Borel MCGs---yet the known ones are of a very special form, and it is completely unknown when one can extend a \emph{given} countable cofinitary group to a maximal \emph{Borel} one.
Note that in $L$, every countable group can be enlarged to a $\mathbf{\Pi}^1_1$ MCG. 
\item It is not even known whether a countable group which cannot be enlarged to a $\mathbf{\Pi}^1_1$ MCG exists in a forcing extension. Even the following is not known: Is it consistent with $\ZFC$ (or $\ZF$) that there exists a countable cofinitary group which cannot be enlarged to a Borel MCG?
\item Is there a model where $2^\omega > \omega_1$ and every uncountable cofinitary group $\mathcal G_0$ of size $<2^\omega$ is a subgroup of a MCG \emph{of the same size} as $\mathcal G_0$? The analogue question for MAD families  is also open; the latter question was originally posed by Fuchino, Geschke, Guzman and Soukup (see~\cite{OGSFSGLS}).

\item Suppose  there is a $\Sigma^1_2$ MCG of size $<\mathfrak c$.
Is there a $\Pi^1_1$ MCG of size $<\mathfrak c$?
Without the requirement on the cardinalities of the groups the question becomes trivial, since there exists a Borel MCG of size $\mathfrak c$.
For MAD families the answer is positive, as follows from a result of the fourth author (see~\cite{AT}). 

\item Is there a model where there is a MCG of size $\alpha$ with $\omega_1<\alpha<2^\omega$ but 
there is no MED family of size $\alpha$?
\end{enumerate}
For many of the above questions, it is also interesting to restrict the question to \emph{projective} MCG, MED families, etc.

\end{document}